\documentclass[11pt,reqno]{amsart}
\usepackage[utf8]{inputenc}

\usepackage{graphicx}
\usepackage{amsmath}
\usepackage{amsthm}
\usepackage{thmtools}
\usepackage{amssymb}
\usepackage{mathrsfs}
\usepackage{siunitx}
\usepackage{tabularx}
\usepackage{booktabs}
\usepackage{bbm}
\usepackage{tikz-cd}
\usetikzlibrary{calc}
\usepackage{quiver}
\usepackage{adjustbox}
\usepackage{enumitem}
\usepackage{stmaryrd}
\usepackage{float}
\usepackage[plainpages=false]{hyperref}
\hypersetup{hypertexnames=false}
\usepackage[capitalize]{cleveref}
\makeatletter
\newenvironment{innerproof}[1][\proofname]
  {\par\normalfont \topsep6\p@ \@plus6\p@\relax
  \trivlist
  \item[\hskip\labelsep\itshape#1\@addpunct{.}]\ignorespaces}
  {\endtrivlist\@endpefalse}
\makeatother
\counterwithin{equation}{section}

\theoremstyle{definition}
\newtheorem{Definition}{Definition}[section]

\newtheorem{Remark}[Definition]{Remark}
\newtheorem{CatalandRemark}[Definition]{Cataland remark}
\newtheorem{Example}[Definition]{Example}
\newtheorem{Warning}[Definition]{Warning}

\theoremstyle{plain}
\newtheorem{Theorem}[Definition]{Theorem}
\newtheorem{Lemma}[Definition]{Lemma}
\newtheorem{Proposition}[Definition]{Proposition}
\newtheorem{Corollary}[Definition]{Corollary}
\newtheorem{Problem}[Definition]{Problem}
\newtheorem{Conjecture}[Definition]{Conjecture}
\newtheorem*{Claim}{Claim}

\DeclareMathOperator{\Hom}{Hom}
\DeclareMathOperator{\Ext}{Ext}

\DeclareMathOperator{\Coker}{Coker}

\DeclareMathOperator{\Image}{Im}
\DeclareMathOperator{\modulecat}{mod}
\DeclareMathOperator{\rank}{rank}

\DeclareMathOperator{\globaldim}{gl.dim}

\DeclareMathOperator{\pdim}{pdim}

\DeclareMathOperator{\addcat}{add}
\DeclareMathOperator{\Gen}{Gen}
\DeclareMathOperator{\Cogen}{Cogen}

\DeclareMathOperator{\tors}{tors}
\DeclareMathOperator{\krew}{krew}
\DeclareMathOperator{\Fund}{Fund}
\DeclareMathOperator{\interval}{int}
\DeclareMathOperator{\proj}{proj}

\DeclareMathOperator{\extproj}{extproj}
\DeclareMathOperator{\extinj}{extinj}
\DeclareMathOperator{\wide}{wide}
\DeclareMathOperator{\rad}{rad}

\newcommand{\identity}{\textnormal{id}}

\newcommand{\bN}{\mathbb{N}}

\newcommand{\bS}{\mathbb{S}}

\newcommand{\bZ}{\mathbb{Z}}

\title{Cambrian lattices are fractionally Calabi-Yau via 2-cluster combinatorics}
\author{Markus Kleinau}
\address{Mathematical Institute of the University of Bonn, Endenicher Allee 60, 53115 Bonn, Germany}
\email{mkleinau@math.uni-bonn.de}
\thanks{Supported by the Deutsche Forschungsgemeinschaft
(DFG, German Research Foundation) under Germany's Excellence Strategy - GZ 2047/1, Projekt-ID
390685813\\}

\begin{document}

\begin{abstract}
    Reading constructed a Cambrian lattice $C_\Gamma$ for each oriented finite type Coxeter diagram $\Gamma$. We show that the derived category of representations of $C_\Gamma$ is fractionally Calabi-Yau for any $\Gamma$, confirming a conjecture of Chapoton. This extends a result of Rognerud for Cambrian lattices of type $A$ with linear orientation, better known as Tamari lattices. If $\Gamma$ is crystallographic, then $C_\Gamma$ is given by the lattice of torsion classes of any hereditary algebra $\Lambda$ of type $\Gamma$. In this case we introduce and study a class of intervals in $C_\Gamma$ whose combinatorics matches the combinatorics of $2$-cluster tilting objects in the 2-cluster category of $\Lambda$. This allows us to compute the Calabi-Yau dimension of $C_\Gamma$.

\end{abstract}

\maketitle

\setcounter{tocdepth}{1}
\tableofcontents

\section{Introduction}
    \subsection{Fractionally Calabi-Yau lattices}
        A \emph{Serre functor} on a category $\mathcal C$ is an autoequivalence $\bS$ that admits a natural isomorphism $\Hom_\mathcal C (X,Y) \cong \Hom_\mathcal C(Y, \bS X)^*$. If a Serre functor exists then it is unique up to natural isomorphism. A triangulated category is called $d$\emph{-Calabi-Yau} for $d\in \bZ$ if the $d$-fold shift $[d]$ is a Serre functor. Examples are given by the derived category of coherent sheaves on a Calabi-Yau variety and cluster categories of Dynkin quivers. More generally a category is called \emph{fractionally Calabi-Yau of dimension} $d/c$ if $\bS^c \cong [d]$. Examples include the derived categories of Dynkin quivers and some components of semiorthogonal decompositions of Fano varieties \cite{AlgGeo}. A recent conjecture of Chapoton in \cite{Chapoton} predicts that the Fukaya-Seidel categories of certain quasi-homogeneous singularities are equivalent to derived categories of lattices. This should imply that the latter are fractionally Calabi-Yau. Motivated by this the fractionally Calabi-Yau property has been shown for Tamari lattices by Rognerud in \cite{RognerudTamari} and for the lattices of order ideals of a product of two chains by Gottesman in \cite{Tal}. In both cases the Calabi-Yau dimension matches the one predicted by Chapoton and in the latter the equivalence to a Fukaya-Seidel category has been shown by combining results of Gottesman \cite{Tal} and Di Dedda \cite{Di_Dedda}. In this article we show that the derived categories of Cambrian lattices are fractionally Calabi-Yau, further corroborating the conjecture of Chapoton.\\
        Let $L$ be a lattice and $k$ a field. A \emph{representation of} $L$ is given by assigning a $k$-vector space $M_a$ to each element $a\in L$ and a linear map $M_{a < b}\colon M_a\rightarrow M_b$ to each relation $a < b$ in $L$ such that $M_{b < c}M_{a < b} = M_{a < c}$. A number of connections between the combinatorics of $L$ and the homological algebra of representations of $L$ have recently been discovered, see for example \cite{AntichainResolution} and \cite{KMT}. We call a lattice $L$ \emph{fractionally Calabi-Yau} if $\mathcal D^b(kL)$, the derived category of the category of representations of $L$, is fractionally Calabi-Yau.\\
        Given an interval $I = [a\leq b]$ in $L$ the \emph{interval module} $M_I$ is given by setting $M_{I,x} = k$ for $x\in I$ and $0$ else. The maps are $\identity_k$ where possible. The interval representations are often well behaved homologically. One approach to show that a lattice is fractionally Calabi-Yau is to exhibit a subset of the interval modules which the Serre functor permutes up to shift. If this subset is big enough then one can apply results of Rognerud \cite[Theorem 1.2]{RognerudTamari} and Gottesman \cite[Theorem B]{Tal} to conclude that the lattice is fractionally Calabi-Yau. Often one can construct a combinatorial model for this subset of interval modules and the permutation on them given by the Serre functor. Computing the Calabi-Yau dimension then becomes a purely combinatorial problem. This approach has been successfully applied to the Tamari Lattices by Rognerud \cite{RognerudTamari} and to the lattice of order ideals of a product of two chains by Gottesman \cite{Tal}. In this article we apply it to the Cambrian lattices. The set of representations we use are given by a new class of intervals we call mutable intervals and the Serre functor on this set can be modeled using the combinatorics of 2-cluster tilting objects in 2-cluster categories.
    \subsection{Cambrian lattices}
        Reading \cite{ReadingCambrian} introduced a Cambrian lattice for each oriented finite type Coxeter diagram $\Gamma$. They provide orientations on the 1-skeleton of the type $\Gamma$-associahedron and hence the mutation graph of type $\Gamma$-cluster algebras. We refer to \cite{CambrianSurvey} for an overview. The Cambrian lattices of type $A$ with linear orientation are exactly the Tamari lattices. For crystallographic $\Gamma$ the Cambrian lattice of type $\Gamma$ is the lattice of torsion classes of any hereditary algebra of type $\Gamma$ as show by Ingalls and Thomas in \cite{IngallsThomas}. This is the perspective we will use in this paper. Let $\Lambda$ be a finite dimensional connected hereditary representation finite algebra. A full additive subcategory $\mathcal T$ of $\modulecat \Lambda$ is a \emph{torsion class} if it is closed under extensions and quotient objects. The torsion classes of $\modulecat \Lambda$ form a lattice under inclusion denoted by $\tors \Lambda$. A second important class of subcategories are the \emph{wide subcategories}. That is the full additive subcategories  $W$ closed under extensions, kernels and cokernels. In the language of Coxeter groups they correspond to noncrossing partitions. The paper \cite{IngallsThomas} established a bijection between the wide subcategories of $\modulecat \Lambda$ and the torsion classes in $\modulecat \Lambda$ given by sending a wide subcategory $W$ to the torsion class $\Gen W$ it generates. We call an interval $[\mathcal T \leq \mathcal T']$ in $\tors \Lambda$ \emph{mutable} if it is of the form $[\Gen W \leq \Gen W']$ for two wide subcategories $W\subseteq W'$. For example all intervals of the form $[0\leq \mathcal T], [\mathcal T\leq \mathcal T]$ or $[\mathcal T \leq \modulecat \Lambda]$ are mutable. For Tamari lattices, that is the lattices of torsion classes of linear type $A$ algebras, the mutable intervals agree with Rognerud's notion of exceptional intervals. This is the desired class of intervals that the Serre functor permutes up to shift.
        \begin{Proposition}
             Let $\Lambda$ be a finite dimensional connected hereditary representation finite algebra and let $I$ be a mutable interval in $\tors(\Lambda)$. Then there exists another mutable interval $\bS I$ and an integer $K_I$ such that $\bS M_I \cong M_{\bS I}[K_I]$.
        \end{Proposition}
        A theorem of Rognerud allows us to glue these objectwise isomorphisms into an equivalence of functors and hence show that $\tors \Lambda$ is fractionally Calabi-Yau. This is our first main result:
        \begin{Theorem}\label{fcy_introduction_tors}
            Let $\Lambda$ be a finite dimensional connected hereditary representation finite algebra. Let $h$ be the Coxeter number of $\Lambda$ and $N$ the number of indecomposable representations of $\Lambda$.
            Then $\tors \Lambda$ is $(2N,2h+2)$-fractionally Calabi-Yau.\\
            Further if $\Lambda$ is of Dynkin type $A_1$, $B_n$, $C_n$, $D_n$ ($n$ even), $E_7$, $E_8$ or $G_2$ then $\tors \Lambda$ is $(N,h+1)$-fractionally Calabi-Yau.
        \end{Theorem}
        In particular all crystallographic Cambrian lattices are fractionally Calabi-Yau. By considering the non-crystallographic types separately, in part by computer, we can extend \cref{fcy_introduction_tors} to all Cambrian lattices. 
        \begin{Theorem}\label{fcy_introduction_Cambrian}
            Let $L$ be an irreducible Cambrian lattice of Coxeter type $\Gamma$. Let $h$ be the Coxeter number of $\Gamma$ and $N$ the number of hyperplanes of the type $\Gamma$ reflection group.
            Then $L$ is $(2N,2h+2)$-fractionally Calabi-Yau.\\
            Further if $\Gamma$ is of Coxeter type  $A_1$, $B_n$, $C_n$, $D_n$ ($n$ even), $E_7$, $E_8$, $H_3$, $H_4$ or $I(m)$ ($m$ even) then $L$ is $(N,h+1)$-fractionally Calabi-Yau.
        \end{Theorem}
        The fractionally Calabi-Yau dimensions claimed here will be explained in the next subsection. For now we just note that they agree with the ones predicted by Chapoton in \cite{Chapoton}. Rognerud computed the action of the Serre functor on exceptional interval modules over the Tamari lattice using explicit projective resolutions. For mutable intervals these resolutions are only available for intervals of the form $[\mathcal{T}\leq \mathcal T]$. To compute the Serre functor for larger intervals we introduce interval mutations. An interval mutation is a triple of mutable intervals $(B,I,A)$ such that
        $I = A\sqcup B$, $\max B = \max I$, and $\min A = \min I$. In terms of modules they correspond to short exact sequences
        \[0\rightarrow M_B\rightarrow M_I\rightarrow M_A\rightarrow 0\]
        and hence to triangles in the derived category. Note that any two parts of an interval mutation determine the third. We can use this and the following proposition to compute the Serre functor inductively.
        \begin{Proposition}\label{mutable_interval_introduction}
            Let $I = [\mathcal T\lneq \mathcal T']$ be a proper mutable interval in $\tors(\Lambda)$. Then there is an interval mutation $(B,I,A)$. In addition for any interval mutation $(B,I,A)$ either $(\bS I,\bS A, \bS B)$ or $(\bS A, \bS B, \bS I)$ is also an interval mutation.
        \end{Proposition}
        In order to explain this phenomenon and to compute the fractionally Calabi-Yau dimension we need a combinatorial model for the set of mutable intervals. This is given by the set of 2-cluster tilting objects in the 2-cluster category of $\Lambda$.
    \subsection{2-cluster combinatorics}
        The cluster category $\mathcal C_1(\Lambda)$ was introduced in \cite{BMRRT} as an additive categorification of the cluster algebra corresponding to $\Lambda$. It is constructed as the orbit category of the derived category $\mathcal D^b(\Lambda)$ under the functor $\tau^{-1}[1]$. An object $T\in \mathcal C_1(\Lambda)$ is called \emph{cluster tilting} if it is maximal with the property $\Ext^1(T,T) = 0$. The set of cluster tilting objects categorifies the clusters of the corresponding cluster algebra in the following way: indecomposable summands correspond to cluster variables, clusters correspond to cluster-tilting objects and mutation corresponds to replacing an indecomposable direct summands with a different one. A generalization of the set of clusters of a cluster algebra to the set of $m$-clusters was introduced by Fomin and Reading in \cite{clusterComplex}. The $m$-clusters also admit an additive categorification using $m$-cluster categories $\mathcal C_m(\Lambda)$ as shown by Thomas in \cite{mclusterThomas} and Zhu in \cite{Zhu}. To construct them one replaces the functor $\tau^{-1}[1]$ in the construction of $\mathcal C_1(\Lambda)$ by $\tau^{-1}[m]$. We refer to \cite{Buan} for an introduction. An object $T\in \mathcal C_m(\Lambda)$ is called $m$\emph{-cluster tilting} if it is maximal with the property $\Ext^i(T,T) =0$ for $1\leq i\leq m$. The $m$-cluster tilting objects also admit a notion of mutation: if one removes an indecomposable direct summand $T'$ from an $m$-cluster tilting object $T$ then there are exactly $m$ other indecomposable objects $T_i$ such that $(T/T') \oplus T_i$ is again an $m$-cluster tilting object. Relevant for this paper is only the case $m=2$. The set of 2-cluster tilting objects of a 2-cluster category has two relevant combinatorial structures: The first is a permutation given by sending a 2-cluster tilting object $T$ to $T[1]$ and the second is the mutation mentioned above relating sets of three 2-cluster tilting objects. We construct a bijection between the set of 2-cluster tilting objects in $\mathcal C_m(\Lambda)$ and the set of mutable intervals in $\tors\Lambda$ which identifies the combinatorial structures on both sets. This is our second main result.
        \begin{Theorem}\label{bijection_introudction}
            Let $\Lambda$ be a finite dimensional connected hereditary representation finite algebra. There is a bijection
            \begin{align*}
               \interval\colon \{\textnormal{2-cluster tilting objects in } \mathcal C_2(\Lambda)\}\rightarrow \{\textnormal{mutable intervals in } \tors(\Lambda)\}
            \end{align*}
            such that for any 2-cluster tilting object $T$ we have $\bS \interval (T) = \interval(T[1])$. Further let $T_1$ be a 2-cluster tilting object with a distinguished indecomposable direct summand $T'$ and  let $T_2$ and $T_3$ be the two 2-cluster tilting objects obtained by mutating $T_1$ at $T'$. Then there is a permutation $\sigma \in  S_3$ such that $(\interval T_{\sigma(1)},\interval T_{\sigma(2)},\interval T_{\sigma(3)})$ is an interval mutation. All interval mutations arise in this way.
        \end{Theorem}
        This reduces the combinatorics of the Serre functor of a Cambrian lattice to the study of the shift functor in 2-cluster categories. The coefficients in \cref{fcy_introduction_tors} are much easier to compute in this setting and \cref{mutable_interval_introduction} can be explained by observing that if we have three 2-cluster tilting objects $T_1$, $T_2$ and $T_3$ connected by mutation then also $T_1[1]$, $T_2[1]$ and $T_3[1]$ are connected by mutation.\\
        A geometric model for mutable intervals in linear type $A$ has been given by Rognerud in \cite{RognerudTamari} in terms of noncrossing trees while a model for 2-cluster tilting objects has been given by Baur and Marsh in \cite{GeometricClusterCategory} in terms of quadrangulations. The bijection $\interval$ also has a simple geometric interpretation in this case.\\
        In the appendix we describe and motivate the problem of classifying periodic Serre formal incidence algebras of lattices, which was one of the original goals when this project started. This problem seems to be a much more tractable problem compared to the classification problem of twisted fractional Calabi-Yau algebras among incidence algebras of lattices. We will give a conjecture for a combinatorial classification via the Coxeter matrix and give an example of a periodic Serre formal incidence algebra of a lattice that is not semidistributive and therefore not a Cambrian lattice.
    
    
    
    The paper is organized as follows:\\
    \textbf{Section 2} gives the necessary background on the Ingalls-Thomas bijections and lattices of torsion classes.\\
    \textbf{Section 3} recalls the 2-cluster category and its 2-cluster tilting objects.\\ 
    \textbf{Section 4} studies the set of mutable intervals and compares them to the set of 2-cluster tilting objects.\\
    \textbf{Section 5} extends this connection to interval mutations and mutation of 2-cluster tilting objects.\\
    \textbf{Section 6} gives an introduction to the relevant representation theory of incidence algebras and the Serre functor.\\
    \textbf{Section 7} completes the proof that Cambrian lattices are periodic Serre formal.\\
    \textbf{Section 8} contains a detailed example in type $A_3$.\\
    \textbf{Section 9} relates this paper to the work of Rognerud on the Tamari lattice and the geometric models for the Tamari lattice and 2-cluster categories.\\
    \textbf{The appendix} discusses and gives a combinatorial conjecture for the classification of periodic Serre formal incidence algebras of lattices via the Coxeter matrix.

\section{The Ingalls-Thomas bijections}
    For an introduction to the representation theory of finite dimensional algebras and the theory of hereditary algebras and torsion classes we refer to the textbook \cite{ASS}.
    Let $\Lambda$ be a finite dimensional connected representation finite hereditary algebra.
    The seminal paper \cite{IngallsThomas} by Ingalls and Thomas collects and extends a set of bijections between many families of objects obtained from the representation theory of $\Lambda$. See also \cite{RingelCatalan} for another account.\\
    This section recalls some of those bijections and some of the theory of lattices of torsion classes. The bijections most relevant to this article are summarised in the following commutative diagram:
    \[\begin{tikzcd}
    	{\{\text{wide subcategories}\}} & {\{\text{torsion classes}\}} & {\{\text{support tilting modules}\}} \\
    	{\{\text{wide subcategories}\}} & {\{\text{torsion-free classes}\}} & {\{\text{support tilting modules}\}}
    	\arrow["\begin{array}{c} W \mapsto \Gen W \\ \mathfrak a (\mathcal T)\mapsfrom \mathcal T \end{array}", tail reversed, from=1-1, to=1-2]
    	\arrow["\begin{array}{c} W \mapsto W^{\perp_{0,1}}\\ {}^{\perp_{0,1}}W \mapsfrom W \end{array}", tail reversed, from=1-1, to=2-1]
    	\arrow["\begin{array}{c} \mathcal T \mapsto \extproj \mathcal T\\ \Gen T \mapsfrom T \end{array}", tail reversed, from=1-2, to=1-3]
    	\arrow["\begin{array}{c} \mathcal T \mapsto \mathcal T^{\perp_0}\\ {}^{\perp_0} \mathcal F \mapsfrom \mathcal F \end{array}", tail reversed, from=1-2, to=2-2]
    	\arrow["\begin{array}{c} W \mapsto \Cogen W \\ \mathfrak a (\mathcal F)\mapsfrom \mathcal F \end{array}"', tail reversed, from=2-1, to=2-2]
    	\arrow["\begin{array}{c} \mathcal F \mapsto \extinj \mathcal T\\ \Cogen T \mapsfrom T \end{array}"', tail reversed, from=2-2, to=2-3]
    \end{tikzcd}\]
\subsection{Subcategories}
    \begin{Definition}
        Let $E\subset \modulecat \Lambda$ be an extension closed full subcategory.
        We call $E$ a
        \begin{enumerate}
            \item \emph{torsion class} if $E$ is closed under taking  quotients,
            \item \emph{torsion-free class} if $E$ is closed under taking submodules,
            \item \emph{wide subcategory} if $E$ is closed under taking kernels and cokernels.
        \end{enumerate}
    \end{Definition}
    Examples of these subcategories can be constructed using orthogonal categories:
    \begin{Definition}
        For a subcategory $\mathcal{C}\subset \modulecat \Lambda$ we define the full subcategories
        \begin{enumerate}
            \item $\mathcal{C}^{\perp_0} = \{Y\in \modulecat \Lambda\mid \Hom_\Lambda(X,Y) = 0\ \forall X\in \mathcal{C}\}$
            \item $\mathcal{C}^{\perp_1} = \{Y\in \modulecat \Lambda\mid \Ext_\Lambda^1(X,Y) = 0\ \forall X\in \mathcal{C}\}$
            \item $\mathcal{C}^{\perp_{0,1}} = \mathcal{C}^{\perp_0}\cap \mathcal{C}^{\perp_1}$
        \end{enumerate}
        The categories ${}^{\perp_0}\mathcal{C}$, ${}^{\perp_1}\mathcal{C}$ and ${}^{\perp_{0,1}}\mathcal{C}$ are defined dually. We also apply these construction to modules by identifying a module with its additive closure.
    \end{Definition}
    \begin{Lemma}
        Let $\mathcal{C}\subset \modulecat \Lambda$ be any subcategory. Then
        \begin{enumerate}
            \item the categories ${}^{\perp_0}\mathcal{C}$ and $\mathcal{C}^{\perp_1}$ are torsion classes,
            \item the categories $\mathcal{C}^{\perp_0}$ and ${}^{\perp_1}\mathcal{C}$ are torsion-free classes and
            \item the categories $\mathcal{C}^{\perp_{0,1}}$ and ${}^{\perp_{0,1}}\mathcal{C}$ are wide subcategories,
        \end{enumerate}
    \end{Lemma}
    \begin{Proposition}
        There are mutually inverse bijections
        \begin{align*}
            \{\text{torsion classes in }\modulecat \Lambda\}&\leftrightarrow \{\text{torsion-free classes in }\modulecat \Lambda\}\\
            \mathcal T&\mapsto \mathcal T^{\perp_0}\\
            {}^{\perp_0}\mathcal F &\mapsfrom \mathcal F
        \end{align*}
        and
        \begin{align*}
            \{\text{wide subcategories in }\modulecat \Lambda\}&\leftrightarrow \{\text{wide subcategories in }\modulecat \Lambda\}\\
            W&\mapsto W^{\perp_{0,1}}\\
            {}^{\perp_{0,1}}W &\mapsfrom W
        \end{align*}
    \end{Proposition}
    For $\mathcal T$ a torsion class we call the pair $(\mathcal T, \mathcal T^{\perp_0})$ a \emph{torsion pair}. Here $\mathcal T$ is the \emph{torsion part} and $\mathcal T^{\perp_0}$ is the \emph{torsion-free part} .
    In order to relate wide subcategories with torsion(-free) classes we need two more constructions.
    \begin{Definition}
        Let $E\subset \modulecat A$ be an extension closed full subcategory. We define the associated wide subcategory $\mathfrak a (E)$ as the additive closure of the set of objects $X$ in $E$ satisfying
        \begin{enumerate}
            \item every morphism in $E$ starting at $X$ has a cokernel in $E$
            \item every morphism in $E$ ending at $X$ has a kernel in $E$
        \end{enumerate}    
    \end{Definition}
    \begin{Definition}
        Let $\mathcal{C}\subset \modulecat \Lambda$ be an additive subcategory. We define the full subcategories
        \[\Gen(\mathcal C) = \{X\in \modulecat \Lambda \mid \exists M\twoheadrightarrow X, M\in \mathcal C  \},\]
        \[\Cogen(\mathcal C) = \{X\in \modulecat \Lambda \mid \exists X\hookrightarrow M, M\in \mathcal C  \}.\]
    \end{Definition}
    \begin{Proposition}
        There are mutually inverse bijections
        \begin{align*}
            \{\text{torsion classes in }\modulecat \Lambda\}&\leftrightarrow \{\text{wide subcategories in }\modulecat \Lambda\}\\
            \mathcal T&\mapsto \mathfrak a(\mathcal T)\\
            \Gen W &\mapsfrom W
        \end{align*}
        and
        \begin{align*}
            \{\text{torsion-free classes in }\modulecat \Lambda\}&\leftrightarrow \{\text{wide subcategories in }\modulecat \Lambda\}\\
            \mathcal F&\mapsto \mathfrak a(\mathcal F)\\
            \Cogen W &\mapsfrom W
        \end{align*}
    \end{Proposition}
    \begin{Proposition}\label{wideRankComplement}
        Let $(\mathcal T,\mathcal F)$ be a torsion pair. Then we have
        \[\mathfrak a (\mathcal T)^{\perp_{0,1}} = \mathfrak a (\mathcal F).\]
    \end{Proposition}
    \begin{Proposition}
        Let $W\subset \modulecat \Lambda$ be a wide subcategory. Then there exists a representation finite hereditary algebra $\Lambda'$ such that $W\cong \modulecat \Lambda'$. In particular $W$ is an abelian category with enough projectives. The $\Ext^1$ functor of $\modulecat \Lambda$ restricts to the $\Ext^1$ functor of $W$. By abuse of notation we will not distinguish between these two functors. 
    \end{Proposition}
    This allows us to apply all constructions from this chapter within a given wide subcategory.
    \begin{Proposition}\label{frakaClosure}
        Let $\mathcal T$ be a torsion class, $M\in \mathfrak a (\mathcal T)$ and $N\subseteq M$ a submodule with $N\in \mathcal T$. Then $N\in \mathfrak a (\mathcal T)$. Dually let $\mathcal F$ be a torsion-free class, $M\in \mathfrak a (\mathcal F)$ and $N$ a factor module of $M$ with $N\in \mathcal F$. Then $N\in \mathfrak a (\mathcal F)$.
    \end{Proposition}
    \begin{Definition}
        The \emph{rank} of a wide subcategory $W$ is the number of direct summands of a basic projective generator.
    \end{Definition}
    For a module $M$ we denote by $\rank(M) = |M|$ the number of indecomposable modules in $\addcat(M)$.
    \begin{Proposition}
        Let $W\subset \modulecat \Lambda$ be a wide subcategory.
        We have
        \[\rank(W) + \rank(W^{\perp_{0,1}}) = \rank (\Lambda)\]
    \end{Proposition}
    \begin{Proposition}
        Let $(\mathcal T,\mathcal F)$ be a torsion pair and $X\in \modulecat \Lambda$.
        Then there exists a canonical short exact sequence
        \[0\rightarrow t(X)\rightarrow X\rightarrow X/t(X)\rightarrow 0\]
        where $t(X)\in \mathcal T$ and $X/t(X)\in \mathcal F$. 
    \end{Proposition}
\subsection{Support tilting modules}
    \begin{Definition}
        The \emph{support} of a module $M\in \modulecat \Lambda$ is the set of indecomposable projectives $P\in \modulecat \Lambda$ with $\Hom_\Lambda(P,M)\neq 0$.
    \end{Definition}
    \begin{Definition}
        A \emph{partial tilting module} $M\in \modulecat \Lambda$ is a basic module satisfying $\Ext_\Lambda^1(M,M) = 0$. A partial tilting module $M$ is \emph{support tilting} if its rank is the same size as its support and \emph{tilting} if it has the same rank as $\Lambda$.
        In this case the \emph{support tilting pair} associated to $M$ is the pair $(M,P_M)$, where $P_M$ is a maximal basic projective module with $\Hom_\Lambda(P_M,M) = 0$.
    \end{Definition}
    In particular for any support tilting pair $(M,P_M)$ we have $|M| + |P_M| = |\Lambda|$. There are two canonical ways to complete a partial tilting module to a support tilting module.
    \begin{Proposition}[Bongartz complement]
        Let $M$ be a partial tilting modules. Then there exist basic modules $\beta(M)\in \Cogen M$ and $\gamma(M)\in \Gen M$ such that $M\oplus \beta(M)$ and $M\oplus \gamma(M'')$ are support tilting modules. The module $\beta(M)$ is called the \emph{Bongartz complement} and the module $\gamma(M)$ is called the \emph{co-Bongartz complement}.
    \end{Proposition}
    In the proposition we already implicitly associated a support tilting module $T$ with its torsion class $\Gen T$ and its torsion-free class $\Cogen T$. The next definition provides a way to go the other way.
    \begin{Definition}
        Let $\mathcal T$ be a torsion class. An object $X\in \mathcal T$ is \emph{ext-projective} in $\mathcal{T}$ if $\Ext_\Lambda^1(X,Y) = 0$ for all $Y\in \mathcal T$. Dually an object $X$ in a torsion-free class $\mathcal F$ is \emph{ext-injective} in $\mathcal F$ if $\Ext_\Lambda^1(Y,X) = 0$ for all $Y\in \mathcal F$.
    \end{Definition}
    \begin{Proposition}
        There are mutually inverse bijections
        \begin{align*}
            \{\text{torsion classes in }\modulecat \Lambda\}&\leftrightarrow \{\text{support tilting modules in }\modulecat \Lambda\}\\
            \mathcal T&\mapsto \extproj(\mathcal T)\\
            \Gen M &\mapsfrom M
        \end{align*}
        where $\extproj(\mathcal T)$ is the direct sum of all ext-projective objects in $\mathcal T$.
    \end{Proposition}
    The dual statement is the following:
    \begin{Proposition}
        There are mutually inverse bijections
        \begin{align*}
            \{\text{torsion-free classes in }\modulecat \Lambda\}&\leftrightarrow \{\text{support tilting modules in }\modulecat \Lambda\}\\
            \mathcal F&\mapsto \extinj(\mathcal F)\\
            \Cogen M &\mapsfrom M
        \end{align*}
        where $\extinj(\mathcal F)$ is the direct sum of all ext-injective objects in $\mathcal F$.
    \end{Proposition}
    The next definition helps describe $\mathfrak a(\Gen T)$ for a support tilting module $T$.
    \begin{Definition}
        Let $\mathcal T$ be a torsion class. An object $X\in \mathcal T$ is \emph{split-projective} in $\mathcal{T}$ if every surjection $Y\twoheadrightarrow X$ in $\mathcal T$ splits. Dually an object $X$ in a torsion-free class is \emph{split-injective} in $\mathcal{F}$ if every injection $X\hookrightarrow Y$ in $\mathcal F$ splits.
    \end{Definition}
    All split-projectives are ext-projective and all split-injectives are ext-injective.
    \begin{Proposition}\label{wideSubcatsViaSplit}
        Let $T$ be a support tilting module. We write $T = T_{split}\oplus T_{ext}$ where $T_{split}$ is the sum of the split-projectives in $\Gen T$ and $T_{ext}$ is the direct sum of the remaining summands of $T$. Then
        \[\mathfrak a (\Gen T) = \{X\in \Gen T \mid \Hom_\Lambda(T_{ext}, X) = 0\}\]
        and $T_{split}$ is the direct sum of the projective objects in $\mathfrak a (\Gen T)$. Dually we can write $T = T'_{split}\oplus T'_{ext}$ where $T'_{split}$ is the sum of the split-injectives in $\Cogen T$ and $T'_{ext}$ is the direct sum of the remaining summands of $T$. Then
        \[\mathfrak a (\Cogen T) = \{X\in \mathcal \Cogen T \mid \Hom_\Lambda(X, T'_{ext}) = 0\}\]
        and $T_{split}$ is the direct sum of the injective objects in $\mathfrak a (\Cogen T)$.
    \end{Proposition}
    \begin{Proposition}\label{WideFromTilting}
        Let $T$ be a partial tilting module and let $W(T)$ be the smallest wide subcategory containing $T$. Then $|T| = \rank(W(T))$. In addition $T$ is tilting in $W(T)$ and $T^{\perp_{0,1}}=W(T)^{\perp_{0,1}}$.
    \end{Proposition}
    
\subsection{Lattices}
    \begin{Definition}
        The lattice $\wide\Lambda$ consists of the wide subcategories of $\modulecat \Lambda$ ordered by inclusion. The meet is given by the intersection of wide subcategories. The lattice $\tors\Lambda$ consists of all torsion pairs $(\mathcal{T},\mathcal{F})$ of $\modulecat \Lambda$. It is ordered by inclusion on the torsion part or equivalently reverse inclusion of the torsion-free part. Here the meet is given by intersection of the torsion part and the join is given by intersection of the torsion-free part. We will often denote an element of $\tors \Lambda$ using only its torsion part.
    \end{Definition}
    \begin{CatalandRemark}
        One can associate an oriented Dynkin diagram to $\Lambda$. Then the lattice $\tors \Lambda$ is isomorphic to the Cambrian lattice associated to that oriented Dynkin diagram and the lattice $\wide \Lambda$ is isomorphic to the lattice of noncrossing partitions.
    \end{CatalandRemark}
    \begin{Lemma}
        The map 
        \begin{align*}
            \wide\Lambda&\rightarrow \tors \Lambda\\
            W&\mapsto (\Gen W,(\Gen W)^{\perp_0})
        \end{align*}
        is a monotone bijection, but in general not an isomorphism.
    \end{Lemma}
    \begin{Definition}
        Let $\mathcal T \lessdot \mathcal T'$ be a cover relation in $\tors \Lambda$. Then there exists a unique object $M = \gamma(\mathcal T \lessdot \mathcal T')$ called the \emph{edge label} of $\mathcal T \lessdot \mathcal T'$ such that $M\in \mathcal T^{\perp_0} \cap\mathcal T'$
    \end{Definition}
    The edge labels can also be described as follows:
    \begin{Proposition} \label{edgeLabelViaJoinIrred}
        Let $M$ be an object in $\modulecat \Lambda$. Then $M$ defines the two torsion classes $\Gen M$ and $\Gen M\setminus\{M\}$.  The edge label of a cover relation $\mathcal T \lessdot \mathcal T'$ is $M$ if and only if $(\Gen M\setminus\{M\})\subset \mathcal T$, $M\notin \mathcal T$ and $M\in \mathcal T'$.
    \end{Proposition}
    Edge labels are also closely related to the wide subcategories associated to torsion classes
    \begin{Proposition}[5.1.8 in \cite{SimplesAreLabels}]\label{SimplesAreLabels}
        Let $\mathcal T'\in \tors \Lambda$ be a torsion class. Then there is a bijection
        \begin{align*}
            \{\mathcal{T}\in \tors \Lambda \mid \mathcal T \lessdot \mathcal T'\}&\rightarrow \{\textnormal{simple objects of }\mathfrak a(\mathcal T')\}\\
            \mathcal T&\mapsto \gamma(\mathcal T \lessdot \mathcal T')\\
            \mathcal T \wedge {}^{\perp_0}S &\mapsfrom S
        \end{align*}
        Dually let $(\mathcal{T},\mathcal F)$ be a torsion pair in $\tors \Lambda$. Then there is a bijection
        \begin{align*}
            \{(\mathcal{T}',\mathcal F')\in \tors \Lambda \mid (\mathcal{T},\mathcal F) \lessdot (\mathcal{T}',\mathcal F')\}&\rightarrow \{\textnormal{simple objects of }\mathfrak a(\mathcal F)\}\\
            (\mathcal{T}',\mathcal F')&\mapsto \gamma((\mathcal{T},\mathcal F) \lessdot (\mathcal{T}',\mathcal F')).\\
            (\mathcal T, \mathcal F)\vee (\Gen S, S^{\perp_0})&\mapsfrom S
        \end{align*}
    \end{Proposition}
    When two torsion classes are related by a cover relation, then the corresponding support tilting modules are related by a mutation, that is they differ by at most one summand.
    \begin{Proposition}\label{mutation}
        Let $\mathcal T \lessdot \mathcal T'$ be a cover relation. Then one of the following hold:
        \begin{enumerate}
            \item $W(\extproj \mathcal T) = W(\extproj \mathcal T')$ and there are indecomposable objects $T_1,T_2$ and a partial tilting object $T$ such that we can write $\extproj \mathcal T \cong T \oplus T_1$ and $\extproj \mathcal T' \cong T \oplus T_2$,
            \item $W(\extproj \mathcal T) \subsetneq W(\extproj \mathcal T')$ and there is an indecomposable object $T_2$ such that $\extproj \mathcal T' \cong \extproj \mathcal T \oplus T_2$.
        \end{enumerate}
    \end{Proposition}
    An important class of intervals in $\tors \Lambda$ is given by the wide intervals. They were introduced and studied in detail in \cite{WideIntervals}.
    \begin{Definition}
        An interval $I=[(\mathcal{T},\mathcal{F}) \leq (\mathcal{T}',\mathcal{F}')]$ is called \emph{wide} if $\mathcal{F}\cap \mathcal T'$ is a wide subcategory.
    \end{Definition}
    \begin{CatalandRemark}
        In \cite{ReadingCambrian} the corresponding intervals of Cambrian lattices are called \emph{facial intervals}.
    \end{CatalandRemark}
    Wide intervals have a number of surprising properties and characterizations. The first one we need is a purely lattice theoretic characterization.
    \begin{Definition}
        An interval $[A,B]$  in a lattice $L$ is called \emph{atomic} if
        \[B=A\vee\bigvee \{C\in [A,B]\mid C \text{ covers } A\}.\] It is called \emph{coatomic} if \[A=B\wedge\bigwedge \{C\in [A,B]\mid C \text{ cocovers } B\}.\]
    \end{Definition}
    \begin{Theorem}[{\cite[Theorem 1.6]{WideIntervals}}]\label{atomicIsWide}
        For an interval $I$ in $\tors \Lambda$ the following are equivalent:
        \begin{enumerate}
            \item $I$ is wide
            \item $I$ is atomic
            \item $I$ is coatomic.
        \end{enumerate}
    \end{Theorem}
    Next we state a description of the wide subcategory associated to a wide interval in terms of the wide subcategories associated to the bounds of the interval.
    \begin{Proposition}[{\cite[Theorem 6.6]{WideIntervals}}]\label{simpleWideIntervals}
        Let $I=[(\mathcal{T},\mathcal{F}) \leq (\mathcal{T}',\mathcal{F}')]$ be a wide interval and $W=\mathcal F\cap \mathcal T'$ the associated wide subcategory. Then we have 
        \begin{enumerate}
            \item $W = \mathfrak a(\mathcal F)\cap \mathfrak a (\mathcal T')$, 
            \item The simple objects of $W$ are a subset of the simple objects of $\mathfrak a(\mathcal F)$ and a subset of the simple objects of $\mathfrak a(\mathcal T')$.
        \end{enumerate}
    \end{Proposition}
    Finally we give a description of the interval as a lattice of torsion classes.
    \begin{Theorem}[{\cite[Theorem 1.6]{WideIntervals}}]\label{wideIntervalsAsTors}
        Let $I=[(\mathcal{T},\mathcal{F}) \leq (\mathcal{T}',\mathcal{F}')]$ be a wide interval of $\tors \Lambda$ and $W=\mathcal F\cap \mathcal T'$ the associated wide subcategory. Then we have a bijection
        \begin{align*}
            \tors W &\rightarrow  I\\
            \hat{\mathcal T}&\mapsto \mathcal T\vee \Gen(\hat {\mathcal T})\\
            \mathcal T \cap W &\mapsfrom \mathcal T
        \end{align*}
        This bijection preserves edge labels. 
    \end{Theorem}

\section{The 2-cluster category}
    Cluster categories were introduced in \cite{BMRRT} as an additive categorification of cluster algebras. They were generalised by Thomas \cite{mclusterThomas} and Zhu \cite{Zhu} to $m$-cluster categories in order to categorify the $m$-cluster complex introduced by Fomin and Reading in \cite{clusterComplex}. We refer for example to \cite{Buan} for an introduction to higher cluster categories and their properties.
    In this section we recall this theory specialised to 2-cluster categories.
    \begin{Definition}
        The 2-cluster category is the orbit category $\mathcal{C}_2(\Lambda) = \mathcal{D}^b(\Lambda)/(\tau^{-1}[2])$. It has a triangulated structure such that the projection functor $\mathcal{D}^b(\Lambda)\rightarrow \mathcal{C}_2(\Lambda)$ is triangulated.
    \end{Definition}
    That $\mathcal{C}_2(\Lambda)$ admits a triangulated structure compatible with the projection was shown by Keller in \cite{Keller}. We can describe objects in the 2-cluster category as follows:
    \begin{Definition}
        Let $\Fund$ be the smallest full additive subcategory of $\mathcal{D}^b(\Lambda)$ containing $\modulecat \Lambda$, $(\modulecat \Lambda)[1]$ and the object $P[2]$ for each projective object $P\in \modulecat \Lambda$. Then $\Fund$ is a fundamental region for the action of $\tau^{-1}[2]$. We use the same notation  for objects in $\Fund$ and their images in $\mathcal{C}_2(\Lambda)$.
    \end{Definition}
    Unlike the case of 1-cluster categories, $\Hom$ and $\Ext^1$ do not change when passing from the module category to the 2-cluster category. This is not true for higher or negative $\Ext^i$.
    \begin{Proposition}\label{homscoincide}
        Let $A,B\in \modulecat \Lambda\subset \Fund$. Then \[\Hom_\Lambda(A,B) = \Hom_{\mathcal{C}_2(\Lambda)}(A,B)\]
        and
        \[\Ext^1_\Lambda(A,B) = \Ext^1_{\mathcal{C}_2(\Lambda)}(A,B).\]
    \end{Proposition}
    \begin{proof}
        By definition we have
        \[\Hom_{\mathcal{C}_2(\Lambda)}(A,B) = \bigoplus_{n\in\bZ}\Hom_{\mathcal D(\Lambda)}(A,\tau^{-n}B[2n])\]
        and
        \[\Ext^1_{\mathcal{C}_2(\Lambda)}(A,B) = \bigoplus_{n\in\bZ}\Hom_{\mathcal D(\Lambda)}(A,\tau^{-n}B[2n+1]).\]
        Consider an indecomposable object $M\in \mathcal 
        D(\Lambda)$. Standard result about the derived category of representation finite hereditary algebras (see \cite{Happel}) imply
        \begin{enumerate}
            \item $M\in\modulecat \Lambda [k]$ for some $k$,
            \item $\tau^{-1} M\in(\modulecat \Lambda [k]\cup \modulecat \Lambda [k+1])$ for $k$ as above and
            \item if $\Hom_{\mathcal D(\Lambda)}(A,M)\neq 0$, then $M\in (\modulecat \Lambda \cup \modulecat \Lambda[1])$.
        \end{enumerate}
        In particular the family of objects $\{\tau^{-n}B[2n]\mid n\in \bZ\}$ contains only one element in $(\modulecat \Lambda \cup \modulecat \Lambda[1])$, namely $B$ itself. Similarly $B[1]$ is the only element of the family $\{\tau^{-n}B[2n+1]\mid n\in \bZ\}$ in $(\modulecat \Lambda \cup \modulecat \Lambda[1])$. The Proposition follows.
    \end{proof}
    The identification of $\tau$ with $[2]$ turns the Auslander-Reiten formula into a Calabi-Yau property:
    \begin{Proposition}[\cite{Keller}]\label{clustercalabiyau}
        The 2-cluster category is 3-Calabi-Yau.
        That is for any $X,Y\in \mathcal{C}_2(\Lambda)$ and $i\in \bZ$ we have 
        \[\Ext_{\mathcal{C}_2(\Lambda)}^i(X,Y) \cong \Ext_{\mathcal{C}_2(\Lambda)}^{3-i}(Y,X)^*\]
    \end{Proposition}
    In particular $\Ext_{\mathcal{C}_2(\Lambda)}^3(M,M) = \Hom_{\mathcal{C}_2(\Lambda)}(M,M)^*$ is zero if and only if $M$ is zero. This motivates the next definition.
    \begin{Definition}
        A basic object $T\in \mathcal{C}_2(\Lambda)$ is
        \begin{enumerate}
            \item \emph{2-rigid} if $\Ext_{\mathcal{C}_2(\Lambda)}^1(T,T) = 0 = \Ext_{\mathcal{C}_2(\Lambda)}^2(T,T)$,
            \item a \emph{2-cluster tilting object} if it is a maximal 2-rigid object or equivalently if it is a 2-rigid object with exactly $n$ summands where $n=|\Lambda|$,
            \item an \emph{almost 2-cluster tilting object} if it is 2-rigid and has exactly $n-1$ summands.
        \end{enumerate}
    \end{Definition}
    An example of a 2-cluster category and (almost) 2-cluster tilting objects can be found in \cref{ExampleSection}.
    The set of 2-cluster tilting objects and its combinatorics play a central role in this paper. There are two important structures on this set: a canonical permutation and a notion of mutation.
    \begin{Definition}
        The \emph{shift permutation} is given by
        \begin{align*}
            [1]\colon\{\text{2-cluster tilting objects}\}&\rightarrow \{\text{2-cluster tilting objects}\}\\
            T&\mapsto T[1]
        \end{align*}
    \end{Definition}
    We state the following enumerative result for later reference. It follows directly from Happel's description of the derived category of $\Lambda$, see \cite{Happel}.
    \begin{Proposition}\label{[1]count}
        Let $h$ be the Coxeter number of $\Lambda$ and $N$ the number of indecomposable representations of $\Lambda$.
        The shift permutation is $2h+2$-periodic and for any 2-cluster tilting object $T$ we have
        \[\sum_{i=1}^{2h+2} \rank (T[i]\cap \modulecat \Lambda) = 2N.\]
        Further if $\Lambda$ is of Dynkin type $A_1$, $B_n$, $C_n$, $D_n$ ($n$ even), $E_7$, $E_8$ or $G_2$, then the shift permutation is $h+1$-periodic and for any 2-cluster tilting object $T$ we have
        \[\sum_{i=1}^{h+1} \rank (T[i]\cap \modulecat \Lambda) = N.\]
    \end{Proposition}
    Next we introduce mutation. In the 1-cluster category every almost cluster tilting object is a direct summand of exactly two cluster tilting object and hence defines a 'mutation' between them. This can be adapted to the 2-cluster category as follows:
    \begin{Proposition}[{\cite[Theorem 3]{mclusterThomas}}]
        Any almost 2-cluster tilting object is a direct summand of exactly 3 non-isomorphic 2-cluster tilting objects.
    \end{Proposition}
    In particular almost 2-cluster tilting objects define a 'mutation' involving three objects instead of two.
    This allows us to relate the number of almost 2-cluster tilting objects with the number of 2-cluster tilting objects.
     \begin{Proposition}\label{almostCounting}
        Let $n$ be the rank of $\Lambda$. Then we have
        \[|\{\textnormal{almost 2-cluster tilting objects in  }\mathcal{C}_2(\Lambda)\}| = \frac{n}{3}|\{\textnormal{2-cluster tilting objects in  }\mathcal{C}_2(\Lambda)\}|\]
    \end{Proposition}    
    \begin{CatalandRemark}
        The 2-rigid objects in $\mathcal{C}_2(\Lambda)$ form a categorification of the 2-cluster complex. In particular the 2-cluster tilting objects correspond to facets.
    \end{CatalandRemark}
    We also get a phenomenon that is unique to higher cluster categories: The three completions of an almost 2-cluster tilting object admit a cyclic ordering:
    \begin{Proposition}\label{cyclicOrderCompletions}
        Let $T_a, T_b$ and $T_c$ be the three completions of an almost 2-cluster tilting object $T$. Then there exists a cyclic order $\prec$ on $\{T_a,T_b,T_c\}$ such that ($T_i\prec T_j$), that is $T_i$ is the predecessor of $T_j$, if and only if $\Ext^1_{\mathcal{C}_2(\Lambda)}(T_j,T_i)\neq 0$.
    \end{Proposition}
    This definition also yields a cyclic order on the two completions of an almost cluster tilting object in the 1-cluster category, but a cyclic order on two objects is not meaningful.

\section{Mutable intervals and 2-cluster tilting objects}
    In the next two sections we will recover the combinatorics of 2-cluster tilting objects within a class of intervals of the lattice of torsion classes called mutable intervals. As a first step we describe 2-cluster tilting objects as triples of partial tilting modules in $\modulecat \Lambda$ satisfying certain orthogonality conditions.\\
    Any 2-cluster tilting object $T\in \mathcal C_2(\Lambda)$ can be decomposed as $T = T_{free}\oplus T_{tors}[1] \oplus T_{supp}[2]$ where $T_{free}, T_{tors}\in \modulecat \Lambda$ and $T_{supp}\in \proj (\modulecat \Lambda )$. The next Proposition describes when three such modules form a 2-cluster tilting object.
    \begin{Proposition}\label{2clusterTiltviaModulecat}
        Let $T_{free}$ and $T_{tors}$ be two partial tilting modules in $\modulecat \Lambda$ and $T_{supp}$ a basic projective module. Then $T = T_{free}\oplus T_{tors}[1] \oplus T_{supp}[2]$ is 2-rigid in $\mathcal C_2(\Lambda)$ if and only if all of the following hold:
        \begin{enumerate}
            \item $\Hom_\Lambda(T_{tors},T_{free}) = 0 = \Ext^1_\Lambda(T_{tors},T_{free})$
            \item $\Hom_\Lambda(T_{supp},T_{tors}) = 0 = \Hom_\Lambda(T_{supp},T_{free})$
        \end{enumerate}
        Further $T$ is 2-cluster tilting if and only if $|\Lambda| = |T_{free}|+|T_{tors}|+|T_{supp}|$.
    \end{Proposition}
    \begin{proof}
        The object $T$ is 2-rigid if and only if $\Ext^1_{\mathcal C_2(\Lambda)}(T,T) = 0 = \Ext^2_{\mathcal C_2(\Lambda)}(T,T)$. Because $\mathcal C_2(\Lambda)$ is 3-Calabi-Yau (\cref{clustercalabiyau}), we only need to check one of the two. Combining \cref{homscoincide} and \cref{clustercalabiyau} yields the following isomorphisms:
        \begin{align}
            \Ext^1_{\mathcal C_2(\Lambda)}(T_{free},T_{free})&\cong \Ext^1_\Lambda(T_{free},T_{free})\label{modulechareq1}\\
            \Ext^1_{\mathcal C_2(\Lambda)}(T_{free},T_{tors}[1])&\cong \Ext^1_\Lambda(T_{tors},T_{free})^*\label{modulechareq2}\\
            \Ext^1_{\mathcal C_2(\Lambda)}(T_{free},T_{supp}[2])&\cong \Hom_\Lambda(T_{supp},T_{free})^*\label{modulechareq3}\\
            \Ext^1_{\mathcal C_2(\Lambda)}(T_{tors}[1],T_{free})&\cong \Hom_\Lambda(T_{tors},T_{free})\label{modulechareq4}\\
            \Ext^1_{\mathcal C_2(\Lambda)}(T_{tors}[1],T_{tors}[1])&\cong \Ext^1_\Lambda(T_{tors},T_{tors})\label{modulechareq5}\\
            \Ext^1_{\mathcal C_2(\Lambda)}(T_{tors}[1],T_{supp}[2])&\cong \Ext^1_\Lambda(T_{supp},T_{tors})^*\label{modulechareq6}\\
            \Ext^1_{\mathcal C_2(\Lambda)}(T_{supp}[2],T_{free})&\cong \Ext^1_\Lambda(T_{free},\nu_\Lambda T_{supp})\label{modulechareq7}^*\\
            \Ext^1_{\mathcal C_2(\Lambda)}(T_{supp}[2],T_{tors}[1])&\cong \Hom_\Lambda(T_{supp},T_{tors})\label{modulechareq8}\\
            \Ext^1_{\mathcal C_2(\Lambda)}(T_{supp}[2],T_{supp}[2])&\cong \Ext^1_\Lambda(T_{supp},T_{supp}).\label{modulechareq9}
        \end{align}
        In \eqref{modulechareq7} the functor $\nu_\Lambda$ is the Nakayama functor and we used the identity $\tau P[1] \cong \nu_\Lambda P$ for $P$ projective.
        The equations \eqref{modulechareq1},\eqref{modulechareq5} and \eqref{modulechareq9} correspond to the assumption that $T_{tors}$, $T_{free}$ and $T_{supp}$ are partial tilting over $\Lambda$. Equation \eqref{modulechareq6} vanishes since $T_{supp}$ is projective. The object $\nu_\Lambda T_{supp}$ is injective in $\modulecat \Lambda$ and hence equation \eqref{modulechareq7} always vanishes. The remaining four equations \eqref{modulechareq2},\eqref{modulechareq3},\eqref{modulechareq4} and \eqref{modulechareq8} match the assumptions. The claim about 2-cluster tilting objects is immediate from the definition.
    \end{proof}
    This decomposition can be used to associate an interval in $\tors \Lambda$ to any 2-cluster tilting object $T$. We also associate two wide subcategories to $T$ which will be important in the proofs.
    \begin{Definition}
        Let $T= T_{free}\oplus T_{tors}[1] \oplus T_{supp}[2]$ be a 2-cluster tilting object in $C_2(\Lambda)$. The interval $\interval (T)$ associated to $T$ is given by 
        \[\interval(T) = [(\Gen T_{tors},T_{tors}^{\perp_0})\leq({}^{\perp_0}T_{free},\Cogen T_{free})].\]
        Further we let $W_{free}(T) = W(T_{free})$ and $W_{tors}(T) = W(T_{tors})$ be the smallest wide subcategories containing $T_{free}$, respectively $T_{tors}$. We write $W_{free} = W_{free}(T)$ and $W_{tors} = W_{tors}(T)$ when this is unambiguous.
    \end{Definition}    
    This construction is computed fore some examples in \cref{ExampleSection}. The intervals obtained this way will turn out to be the mutable intervals which we introduce next.
    \begin{Definition}
        Let $ I = [(\mathcal{T},\mathcal{F}) \leq (\mathcal{T}',\mathcal{F}')]$ be an interval in the lattice of torsion pairs. We call $I$ \emph{mutable} if it satisfies $\mathfrak a(\mathcal{T})\subset \mathfrak a(\mathcal{T}')$ or equivalently $\mathfrak a(\mathcal{F}')\subset \mathfrak a(\mathcal{F})$.
    \end{Definition}
    \begin{proof}[Proof of the equivalence of definitions.]
        Assume $\mathfrak a(\mathcal{T})\subset \mathfrak a(\mathcal{T}')$. Then 
        \[\mathfrak a(\mathcal{F}') = \mathfrak a(\mathcal{T}')^{\perp_{0,1}} \subset \mathfrak a(\mathcal{T})^{\perp_{0,1}} = \mathfrak a(\mathcal{F}).\]
    \end{proof}
    \begin{Example}\label{mutableExamples}
        Let $(\mathcal{T},\mathcal{F})$ be a torsion pair. Then the following intervals are always mutable:
        \begin{enumerate}
            \item $[(\mathcal{T},\mathcal{F}) \leq (\mathcal{T},\mathcal{F})]$
            \item $[(\mathcal{T},\mathcal{F}) \leq (\modulecat \Lambda, 0)]$
            \item $[(0,\modulecat \Lambda) \leq (\mathcal{T},\mathcal{F})]$
        \end{enumerate}
    \end{Example}
    More explicit examples in type $A_3$ can be found in \cref{ExampleSection}.
    There are mutable intervals that are not wide and wide intervals that are not mutable. The two properties don't seem to be directly related.
    \begin{Proposition}\label{intiswelldefined}
        Let $T = T_{free}\oplus T_{tors}[1] \oplus T_{supp}[2]$ be a 2-cluster tilting object. Then $\interval (T)$ is a mutable interval.
    \end{Proposition}
    \begin{proof}
        To simplify notation we write $\mathcal F' = \Cogen T_{free}$ and $\mathcal T = \Gen T_{tors}$. We need to show $\mathfrak a(\mathcal{F}') \subset \mathfrak a(\mathcal{T})^{\perp_{0,1}}$. The class $\mathcal F'$ is cogenerated by the torsion-free class $\mathcal{F'}\cap W_{free}$ of $W_{free}$ because $T_{free}\in \mathcal{F'}\cap W_{free}$. This implies $\mathfrak a(\mathcal{F}') = \mathfrak a(\mathcal{F}'\cap W_{free})\subset W_{free}$. Dually we get $\mathfrak a(\mathcal{T}) \subset W_{tors}$. By \cref{WideFromTilting}, $T_{free}\in T_{tors}^{\perp_{0,1}}$ implies $W_{free}\subset W_{tors}^{\perp_{0,1}}$ and hence
        \[\mathfrak a(\mathcal{F}) \subset W_{free}\subset W_{tors}^{\perp_{0,1}} \subset \mathfrak a(\mathcal{T'})^{\perp_{0,1}}.\]
    \end{proof}
    \begin{CatalandRemark}
        Mutable intervals are a generalization of the exceptional intervals Rognerud defined for Tamari lattices in \cite{RognerudTamari}. We use a different name here to avoid confusion with exceptional sequences. While the lattice of torsion classes forms a Cambrian lattice, the lattice of wide subcategories corresponds to noncrossing partitions. From this perspective mutable intervals correspond to intervals in the lattice of noncrossing partitions. In the Tamari case this connection was already observed by Rognerud in \cite{RognerudNCP}. As part of m-eralized Fuß-Catalan combinatorics intervals of noncrossing partitions are better known as 2-noncrossing partitions. Both the 2-noncrossing partitions and the 2-clusters are counted by 2-$W$-Catalan numbers and bijections between them are known.
    \end{CatalandRemark}
    Our next goal is to show that $\interval$ defines a bijection between 2-cluster tilting objects and mutable intervals.
    \begin{Lemma}\label{equiCardinal}
        We have 
        \[|\{\textnormal{2-cluster tilting objects in } \mathcal C_2(\Lambda)\}| = |\{\textnormal{mutable intervals in } \tors(\Lambda)\}|.\]
        In fact both are given by the 2-$W$-Catalan number where $W$ is the Coxeter type of $\Lambda$.
    \end{Lemma}
    \begin{proof}
        Mutable intervals are by definition in bijection with intervals in the lattice of wide subcategories. The lattice of wide subcategories of $\Lambda$ is isomorphic to the lattice of noncrossing partitions of $W$ and intervals in that lattice, better known as 2-noncrossing partitions, are counted by the 2-$W$-Catalan number.\\
        The 2-cluster tilting objects in $\mathcal C_2(\Lambda)$ are in bijection with 2-clusters of type $W$, which are also counted by 2-$W$-Catalan number. For more details on the underlying Fuß-Catalan combinatorics see \cite{Cataland}.
    \end{proof}
    \begin{Proposition}\label{clustersAreMutableBijection}
        There is a bijection
        \begin{align*}
           \interval\colon \{\textnormal{2-cluster tilting objects in } \mathcal C_2(\Lambda)\}&\rightarrow \{\textnormal{intervals in } \tors(\Lambda)\}\\
           T&\mapsto \interval (T)
        \end{align*}
    \end{Proposition}
    \begin{proof}
        The map is well defined by \cref{intiswelldefined}. By \cref{equiCardinal} we only need to show that it is injective. Let $T = T_{free}\oplus T_{tors}[1] \oplus T_{supp}[2]$ be a 2-cluster tilting object. We will reconstruct $T$ from $\interval(T) = [\Gen 
        T_{tors}\leq{}^{\perp_0}\Cogen T_{free}]$. The module $T_{supp}$ is the maximal basic projective module in ${}^{\perp_0}T_{tors} \cap {}^{\perp_0}T_{free} = {}^{\perp_0}\Gen T_{tors} \cap {}^{\perp_0}\Cogen T_{free}$. Let $\gamma(T_{tors})$ be the Bongartz complement of $T_{tors}$. Then $\widetilde{T} =T_{tors}\oplus \gamma(T_{tors})$ is the unique support tilting module generating $\Gen T_{tors}$. In order to identify the summands of $\widetilde{T}$ belonging to $T_{tors}$ we need the following claim:
        \begin{Claim}
            Let $\gamma(T_{tors})$ be the Bongartz complement and $\hat T$ a direct summand of $\gamma(T_{tors})$. Then there exists an indecomposable $I\in \mathfrak{a}(\Cogen T_{free})$ that is injective in $\mathfrak{a}(\Cogen T_{free})$ with $\Ext^1_\Lambda(\hat T, I) \neq 0$.
        \end{Claim}
        All indecomposable injective objects of $\mathfrak{a}(\Cogen T_{free})$ are direct summands of $T_{free}$ and hence $\Ext^1_\Lambda(T_{tors},I)=0$ for every injective object of $\mathfrak{a}(\Cogen T_{free})$. This distinguishes the summands of $\widetilde{T}$ belonging to $T_{tors}$ from those belonging to $\gamma(T_{tors})$ in a way that only depends on $\interval(T)$. Dually we can also reconstruct $T_{free}$ from $\interval(T)$. 
        \begin{innerproof}[Proof of claim]
            All relevant objects live in the wide subcategory $T_{supp}^{\perp_{0,1}}$ and $T_{free}\oplus T_{tors}[1]$ is a 2-cluster tilting object in the 2-cluster category corresponding to $T_{supp}^{\perp_{0,1}}$. So without loss of generality we assume $T_{supp} = 0$.
             By \cref{WideFromTilting} we know that $T_{free}$ and $T_{tors}$ are tilting modules in $W_{free}$, respectively $W_{tors}$. Counting summands yields 
            \[\rank (W_{free}) + \rank (W_{tors}) = |\Lambda|.\]
            Further we have 
            \[\Hom_\Lambda(W_{tors},W_{free}) = 0 = \Ext_\Lambda^1(W_{tors},W_{free})\]
            and hence $W_{free} = W_{tors}^{\perp_{0,1}}$. In particular $(\Gen W_{tors},\Cogen W_{free})$ is a torsion pair.
            By definition $\hat{T}\in \Gen W_{tors}$, but $\hat T \notin W_{tors}$, since otherwise $T_{tors}\oplus \hat T$ would be a partial tilting module in $W_{tors}$ with more summands than the rank of $W_{tors}$. Since $\Hom_\Lambda(\hat T, T_{free}) = 0$ we must have $\Ext_\Lambda^1(\hat T, T_{free}) \neq 0$. We choose an injective object $\hat{I}$ of $\mathfrak{a}(\Cogen T_{free})$ that admits an injection $\phi\colon T_{free}\hookrightarrow \hat{I}$. Then the following is part of a long exact sequence:
            \[\Hom_\Lambda(\hat T,\Coker(\phi))\rightarrow \Ext_\Lambda^1(\hat T,T_{free})\rightarrow \Ext_\Lambda^1(\hat T,\hat I).\]
            We know $\Hom_\Lambda(\hat T,\Coker(\phi)) = 0$, since $\Coker(\phi)\in W_{free}$ and $\hat T\in \Gen W_{tors}$. Then $\Ext_\Lambda^1(\hat T,T_{free}) \neq 0$ implies $\Ext_\Lambda^1(\hat T,\hat I)\neq 0$. The object $I$ from the claim is obtained by choosing a direct summand of $\hat I$ that witnesses the non-vanishing.
        \end{innerproof}
    \end{proof}
    \begin{Remark}
        A different bijection between 2-clusters and 2-noncrossing partitions (and by extension to mutable intervals) has been given in \cite{SiltingVersusSimpleMinded}.
    \end{Remark}
    The next definition constructs an inverse to $\interval$ by associating additional wide subcategories and partial tilting modules to a mutable intervals. This construction is easier to work with than the one used in the proof of \cref{clustersAreMutableBijection}.
    \begin{Definition}
        Let $I=[(\mathcal{T},\mathcal{F}) \leq (\mathcal{T}',\mathcal{F}')]$ be a mutable interval. We let $T_{free}^I$ be the $\mathfrak a(\mathcal F)$-support tilting module corresponding to the torsion-free class $\mathcal F'\cap \mathfrak a(\mathcal F )$ of $\mathfrak a(\mathcal F)$ and $T_{tors}^I$ be the $\mathfrak a(\mathcal T')$-support tilting module corresponding to the torsion class $\mathcal T\cap \mathfrak a(\mathcal T' )$ of $\mathfrak a(\mathcal T')$. Finally we let $T_{supp}^I$ be the maximal basic projective module with $\Hom_\Lambda(T_{supp},\mathcal T) =0 =\Hom_\Lambda(T_{supp},\mathcal F')$. In addition we define the wide subcategories $W_{tors}^I = W(T_{tors}^I)$ and $W_{free}^I = W(T_{free}^I)$.
    \end{Definition}
    In particular we have $\mathfrak a (\mathcal T)\subseteq W^I_{tors} \subseteq \mathfrak a (\mathcal T')$ and $\mathfrak a (\mathcal F')\subseteq W^I_{free} \subseteq \mathfrak a (\mathcal F)$.
    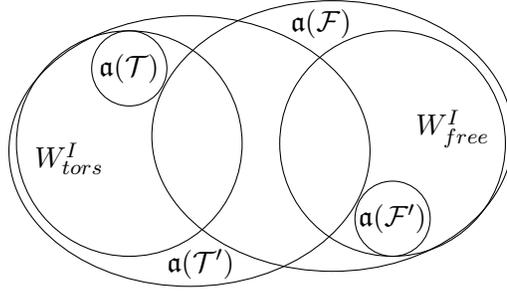
\begin{figure}[h]
        \centering
        \caption{A Venn diagram of the wide categories involved in a mutable pair. There are no morphisms or extension going from a category on the left to one further to the right.}
    \begin{tikzpicture}
        \draw (0,0) circle[radius=1.5] ; 
        \node at (-0.8,-0.2) {$W^I_{tors}$};
		\draw (3.5,0) circle[radius=1.5] ; 
        \node at (4.3,0.2) {$W^I_{free}$};
        \draw (0,1) circle[radius = 0.5] ; 
        \node at (0,1) {$\mathfrak a (\mathcal T)$};
        \draw (3.5,-1) circle[radius = 0.5] ; 
        \node at (3.5,-1) {$\mathfrak a (\mathcal F')$};
        \draw (0.8,-0.1) ellipse[x radius=2.4, y radius = 1.8] ; 
        \node at (0.95,-1.6) {$\mathfrak a (\mathcal T')$};
        \draw (2.7,0.1) ellipse[x radius=2.4, y radius = 1.8] ; 
        \node at (2.55,1.6) {$\mathfrak a (\mathcal F)$};
    \end{tikzpicture}
    \end{figure}
    \begin{Proposition}\label{moduleFromInterval}
        Given a mutable interval $I$ we have
        \[\interval(T_{free}^I\oplus T_{tors}^I[1]\oplus T_{supp}^I[2]) = I\]
    \end{Proposition}
    \begin{proof}
        We let $T'_{free}\oplus T'_{tors}[1]\oplus T'_{supp}[2]$ be the preimage of $I$ under $\interval$ and write $I=[(\mathcal T,\mathcal F)\leq (\mathcal T',\mathcal F')]$.
        The object $T'_{supp}$ is maximal among the basic projective objects $P$ with the property  $\Hom_{\Lambda}(P,T'_{free})=0=\Hom_{\Lambda}(P,T'_{tors})$ while $T_{supp}^I$ is maximal with $\Hom_{\Lambda}(P,\Cogen T'_{free})=0=\Hom_{\Lambda}(P,\Gen T'_{tors})$. Since those two properties are equivalent we have $T'_{supp}\cong T_{supp}^I$.\\
        Let $\Tilde{T}$ be the support tilting modules associated to $\mathcal T$. The claim in the proof of \cref{clustersAreMutableBijection} identifies $T'_{tors}$ as containing exactly those summands $X$ of $\Tilde{T}$ with $\Ext_\Lambda^1(X,I)=0$ where $I$ is the injective cogenerator of $\mathfrak a (\mathcal F')$ or equivalently exactly those summands in $\mathfrak a (\mathcal F')^{\perp_{0,1}}=\mathfrak{a}(\mathcal T')$. In particular $T'_{tors}$ is the support tilting module of $\mathcal T\cap \mathfrak a (\mathcal T')$ in $\mathfrak a (\mathcal T')$, matching the definition of $T_{tors}^I$. Dually we get $T'_{free}\cong T_{free}^I$.
    \end{proof}
    Our next goal is to interpret the shift permutation in terms of mutable intervals. We give a description using torsion classes here and one using the Serre functor in \cref{CategoricalSerre}.
    \begin{Definition}
         Let $ I = [(\mathcal{T},\mathcal{F}) \leq (\mathcal{T}',\mathcal{F}')]$ be a mutable interval. The image of $I$ under the \emph{Serre permutation} is  \[\mathbb SI =(\Gen T_{free}^I,T_{free}^{I\perp_0}) \leq (\mathcal{T}\vee \Gen W_{free}^I,\mathcal{F}\cap W_{free}^{I\perp_0}).\] 
         The \emph{inverse Serre permutation} of $I$ is given by
         \[\mathbb S^{-1}I =(\mathcal T'\cap {}^{\perp_0}W_{tors}^I,\mathcal F'\wedge \Cogen W_{tors}^I) \leq ({}^{\perp_0}T_{tors}^I,\Cogen T_{tors}^I).\]
    \end{Definition}
    We need some preparation before comparing the Serre- and shift permutations.
    \begin{Lemma}\label{wideIntfromExcepInt}
        Let $ I = [(\mathcal{T},\mathcal{F}) \leq (\mathcal{T}',\mathcal{F}')]$ be a mutable interval. Then the interval $J=[\mathcal{T}\leq\mathcal T\vee \Gen W_{free}^I]$ is a wide interval with associated wide subcategory $W_{free}^I$. 
    \end{Lemma}
    \begin{proof}
        Let $S_1,\dots,S_k$ be the $W_{free}^I$-simple objects. The module $T_{free}^I$ is a support tilting module in $\mathfrak a (\mathcal F)$. Therefore the simple objects in $W_{free}^I$ are a subset of the simple objects in $\mathfrak a (\mathcal F)$ and by \cref{SimplesAreLabels} we know that the $S_i$ are labels of edges $\mathcal T\lessdot \mathcal T_i$. We compute
        \[\mathcal T\vee \Gen W_{free}^I = \mathcal{T}\vee\bigvee_i\Gen S_i.\]
        This shows that $J$ is atomic and hence wide by \cref{atomicIsWide}. We know from \cref{wideIntervalsAsTors} that the simple objects of the wide subcategory associated to $J$ are exactly the edge labels of cover relations $\mathcal{T}\lessdot \Tilde{\mathcal T}$ with $\Tilde{\mathcal T} < \mathcal T\vee \Gen W_{free}^I$. We have already seen that all of the $S_i$ belong to this set. Let $\hat S$ be a simple object of $\mathfrak a (\mathcal F)$ not in $W_{free}^I$. Then $\hat S\in \mathcal F\cap \bigcap_{i}S_i^{\perp_0} = \mathcal F \cap W_{free}^{I\perp_0}$. In particular $\mathcal T\vee \Gen \hat S$ is not less than $\mathcal T\vee \Gen W_{free}^I$. This shows that the only covers of $\mathcal T$ in $J$ correspond to simple objects in $W_{free}^I$ and hence the Lemma.
    \end{proof}
    \begin{Lemma}\label{SerrePermWellDefined}
        Let $I$ be a mutable interval.
        Then $\bS I$ and $\bS^{-1}I$ are mutable intervals.
    \end{Lemma}
    \begin{proof} 
        By duality it suffices to check $\bS I$. We write $ I = [(\mathcal{T},\mathcal{F}) \leq (\mathcal{T}',\mathcal{F}')]$. We will show \[\mathfrak{a}(\Gen T_{free}^I)  \subseteq W_{free}^I\subseteq\mathfrak{a}(\mathcal{T}\vee W_{free}^I).\]
        "$\mathfrak{a}(\Gen T_{free}^I)  \subseteq W_{free}^I$": The subcategory $W_{free}^I\cap\Gen T_{free}^I$ is a torsion class of $W_{free}^I$ and hence generated by a unique wide subcategory $W\subset W_{free}^I$. We have \[\Gen T_{free}^I = \Gen (W_{free}^I\cap\Gen(T_{free}^I)) =\Gen W\] and hence $\mathfrak a (\Gen T_{free}^I) = W\subseteq W_{free}^I$.\\
        "$W_{free}^I\subseteq\mathfrak{a}(\mathcal{T}\vee \Gen W_{free}^I)$": We know from \cref{wideIntfromExcepInt} that the interval $\mathcal{T}\leq\mathcal T\vee \Gen W_{free}^I$ is wide with associated wide subcategory $W_{free}^I$. Then by \cref{simpleWideIntervals} we get $W_{free}^I\subseteq\mathfrak{a}(\mathcal{T}\vee \Gen W_{free}^I)$.
    \end{proof}
    \begin{Proposition}\label{inverseofserre}
        The maps $\bS$ and $\bS^{-1}$ define mutually inverse bijections on the set of mutable intervals in $\tors \Lambda$.
    \end{Proposition}
    \begin{proof}
        Let $I =[(\mathcal{T},\mathcal{F}) \leq (\mathcal{T}',\mathcal{F}')]$ be a mutable interval. We will show $\bS^{-1}\bS I = I$. To simplify notation we write $\bS I = [(\mathcal{ST},\mathcal{SF}) \leq (\mathcal{ST}',\mathcal{SF}')]$. We start by computing $T_{tors}^{\bS I}$ and $W_{tors}^{\bS I}$. The simple objects in $W_{free}^I$ are exactly the simple objects appearing in a $\mathfrak {a}(\mathcal F)$ composition series of $T_{free}^I$, because $T_{free}^I$ is an $\mathfrak {a}(\mathcal F)$-support tilting module. As observed in the proof of \cref{SerrePermWellDefined}, the simple objects in $W_{free}^I$ are also simple in $\mathfrak a(\mathcal T \vee \Gen W_{free}^I) = \mathfrak a (\mathcal {ST}')$. Therefore $T_{free}^I$ is a support tilting module in $\mathfrak a (\mathcal {ST}')$ with $\Gen T_{free}^I = \mathcal {ST}$. Hence $T_{free}^I = T_{tors}^{\bS I}$ and $W_{free}^I = W_{tors}^{\bS I}$. The torsion-free part of the upper bound of $\bS^{-1}\bS I$ is given by $\Cogen T_{tors}^{\bS I} = \Cogen T_{free}^I = \mathcal F'$ and therefore agrees with the upper bound of $I$. It remains to check that the lower bounds agree. The interval $[\mathcal T\leq \mathcal T\vee \Gen W_{free}^I]$ is wide by \cref{atomicIsWide} with associated wide subcategory $W_{free}^I$. In particular it is coatomic by \cref{atomicIsWide}. We compute on torsion-free parts
        \begin{align*}
            \mathcal F &=  (\mathcal F \cap W_{free}^{I\perp_0})\wedge\bigwedge_{S\text{ simple in }W_{free}^I}\left((\mathcal F \cap W_{free}^{I\perp_0}) \wedge \Cogen S\right)\\
           &= (\mathcal F \cap W_{free}^{I\perp_0})\wedge \Cogen W_{free}^I\\
           &=  \mathcal{SF}\wedge \Cogen W_{tors}^{\bS I}
        \end{align*}
        where the first equality follows from the interval being coatomic. This shows that the lower bounds agree. A dual argument yields $\bS\bS^{-1}I = I$.      
    \end{proof}
    Finally we can identify the Serre permutation with the shift permutation.
    \begin{Proposition}\label{equivariance}
        Let $T$ be a 2-cluster tilting object in $C_2(\Lambda)$. Then 
        \[\interval(T[1]) = \bS\interval(T).\]
    \end{Proposition}
    \begin{proof}
        We define a partial order $\prec$ on the set of mutable intervals by $[A\leq B]\prec [C\leq D]$ if and only if $B\leq D$. Our goal is to show that for every 2-cluster tilting object $T$ we have $\bS\interval(T)\prec\interval(T[1])$. The proposition follows from this because both $\interval(-[1])$ and $\bS\interval(-)$ are bijections.\\
        Let $T$ be  2-cluster tilting object. Using \cref{moduleFromInterval} we can write $T=\interval(T_{free}^I\oplus T_{tors}^I[1]\oplus T_{supp}^I[2]$ where $I = \interval(T)$. Further we write $T[1] = T'_{free}\oplus T'_{tors}[1]\oplus T'_{supp}[2]$. Note that $T'_{tors} \cong T_{free}^I$. The lower bound of $\interval(T[1])$ is given by $\Gen T'_{tors}$ while the lower bound of $\bS\interval(T)$ is given by $\Gen T_{free}^I$. In particular both are equal. The torsion-free part of the upper bound of $\interval(T[1])$ is given by $\Cogen T'_{free}$, while the torsion-free part of the upper bound of $\bS\interval(T)$ is given by $T_{tors}^{I\perp_0}\cap W_{free}^{I\perp_0}.$ That is we have to show 
        \[T'_{free}\in T_{tors}^{I\perp_0}\cap W_{free}^{I\perp_0}\]
        because torsion-free classes are ordered by reverse inclusion. We have $T'_{free}\in T_{tors}^{\prime\perp_{0,1}} = T_{free}^{I\perp_{0,1}} = W_{free}^{I\perp_{0,1}}$ and compute
        \begin{align*}
            \Hom_{\Lambda}(T_{tors}^I,T'_{free}) &= \Hom_{\mathcal C^2(\Lambda)}(T_{tors}^I,T'_{free})\\
            &= \Ext^2_{\mathcal C^2(\Lambda)}(T_{tors}^I[1],T'_{free}[-1])\\
            &\subseteq \Ext^2_{\mathcal C^2(\Lambda)}(T_{tors}^I[1],T_{tors}^I[1]\oplus T_{supp}^I[2])\\
            &=0.
        \end{align*}
        Here the first equality is \cref{homscoincide}, the inclusion holds because the summands of $T[1]$ in $\modulecat\Lambda$ are shifts of summands of $T$ in $\modulecat \Lambda[1]$ and $\proj \Lambda[2]$. The final equality holds because $T$ is 2-cluster tilting. This shows the required inclusion and thus proves the proposition.
    \end{proof}
    \cref{[1]count} gave some enumerative information on the orbits of the shift permutation. The next proposition translates this to mutable intervals. This will allow us to compute the fractionally Calabi-Yau dimension later.
    \begin{Proposition}\label{serrecount}
        Let $h$ be the Coxeter number of $\Lambda$ and $N$ the number of indecomposable representations of $\Lambda$.
        The Serre permutation is $2h+2$-periodic and for any mutable interval $I$ we have
        \[\sum_{i=1}^{2h+2} \rank (W_{free}^{\bS^iI}) = 2N.\]
        Further if $\Lambda$ is of Dynkin type $A_1$, $B_n$, $C_n$, $D_n$ ($n$ even), $E_7$, $E_8$ or $G_2$, then the Serre permutation is $h+1$-periodic and for any mutable interval $I$ we have
        \[\sum_{i=1}^{h+1} \rank (W_{free}^{\bS^iI}) = N.\]
    \end{Proposition}
    We conclude the chapter with an alternative description of the Serre permutation for trivial intervals. The fractionally Calabi-Yau property will be proven by induction on the size of the mutable interval. The next Proposition is critical to the induction start, while the next chapter prepares the induction step.
    \begin{Proposition}\label{serrePermSimple}
        Let $I= [(\mathcal T,\mathcal F)\leq(\mathcal T,\mathcal F)]$ be a simple interval and let $S_1,\dots ,S_{k_I}$ be the simple objects in $\mathfrak a(\mathcal F)$. We define $\hat I$ as the set of all $\hat {\mathcal T}\in  \tors \Lambda$ satisfying
        \begin{enumerate}
            \item $\hat{\mathcal T}\leq \mathcal T \vee \Gen \mathfrak  a(\mathcal F)$ and
            \item for each $1\leq i\leq k_I$ we have $\hat{\mathcal T}\nleq \mathcal T \vee \bigvee_{j\neq i} \Gen  S_j$.
        \end{enumerate}
        Then $\bS I = \hat I$.    
    \end{Proposition}
        \begin{proof}
            In this case $T_{free}^I$ is by definition the injective cogenerator of $\mathfrak a(\mathcal F)$ and $W_{free}^I =\mathfrak a(\mathcal F)$.\\
            "$\bS I\subset \hat I$":
            Recall that $\bS I = [\Gen T_{free}^I\leq \mathcal T \vee \Gen W_{free}^I]$. In particular every element of $\bS I$ satisfies $(1)$. If there is an element of $\bS I$ that does not satisfy $(2)$, then $\Gen T_{free}^I$ does not satisfy $(2)$ either. In this case let $i$ be such that 
                \[\Gen T_{free}^I\leq \mathcal T \vee \bigvee_{j\neq i} \Gen  S_j.\]
            The interval $[\mathcal T \leq \mathcal T \vee \bigvee_{j\neq i} \Gen  S_j]$ is atomic, hence wide with associated wide subcategory $W_i$. The simple objects of $W_i$ are the $S_j$ with $j\neq i$. Let $T_i$ be the injective hull of $S_i$ in $W_{free}^I$. Then there is a short exact sequence
            \[0\rightarrow S_i\rightarrow T_i\rightarrow T_i/S_i\rightarrow 0\]
            in $W_{free}^I$. The module $T_i/S_i$ is filtered by the modules $S_j$ with $j\neq i$ and hence we have $T_i/S_i\in W_i\subset \mathfrak a(\mathcal T \vee \bigvee_{j\neq i} \Gen S_j)$. The module $T_i$ is injective in $W^I_{free}$ and hence a direct summand of $T^I_{free}\in \mathcal T \vee \bigvee_{j\neq i} \Gen  S_j$. The definition of $\mathfrak a$ and the short exact sequence now show $S_i \in \mathcal T \vee \bigvee_{j\neq i} \Gen S_j$, but we also know \[S_i \in \mathcal T^{\perp_0} \vee \bigvee_{j\neq i} S_j^{\perp_0} = (\mathcal T \vee \bigvee_{j\neq i} \Gen S_j)^{\perp_0},\]
            a contradiction.\\
            "$\hat I \subset\bS I$":
            Let $\hat{\mathcal T}\in \hat I$ be a torsion class. Property $(1)$ implies $\hat{\mathcal T} \leq \mathcal T \vee \Gen W_{free}^I$ so it remains to show $\Gen T_{free}^I\leq \hat{\mathcal T}$. We will prove this in two steps:\\
            \textbf{Step 1:} $T_{free}^I\in \mathcal T\vee\hat{\mathcal T}$:\\
            Because $\hat T$ satisfies conditions $(1)$ and $(2)$, so does $\mathcal T\vee\hat{\mathcal T}$. In addition the torsion class $\mathcal T\vee\hat{\mathcal T}$ is in the wide interval $[\mathcal T\leq \mathcal T \vee \Gen W_{free}^I]$. Using \cref{wideIntervalsAsTors} we can interpret this in $\tors W_{free}^I$. We have to show $T_{free}^I\in (\mathcal T\vee\hat{\mathcal T}) \cap W_{free}^I$. For a given $1\leq i \leq k_I$ \cref{wideIntervalsAsTors} shows that the torsion class
            $(\mathcal T \vee \bigvee_{j\neq i} \Gen S_j)\cap W_{free}^I$ of $W_{free}^I$ consist of exactly those objects filtered by the $S_j$ except $S_i$. Condition $(2)$ then shows that $S_i$ appears in the $W_{free}^I$-composition series of some object $X$ of $(\mathcal T\vee\hat{\mathcal T})\cap W_{free}^I$. Let $T_i$ be the $W_{free}^I$-injective hull of $S_i$. Because $S_i$ appears in the composition series, there must be a non-zero homomorphism $f\colon X\rightarrow T_i$ and hence a short exact sequence
            \[0\rightarrow \Image(f)\rightarrow T_i\rightarrow T_i/\Image(f)\rightarrow 0\]
            In particular $T_i$ is in $(\mathcal T\vee\hat{\mathcal T})\cap W_{free}^I$ if and only if $T_i/\Image(f)$ is and the latter is a $W_{free}^I$-injective object of strictly lower $W_{free}^I$-length. By infinite descent we are done since $0\in (\mathcal T\vee\hat{\mathcal T})\cap W_{free}^I$. The object $T_{free}^I$ is the direct sum of all indecomposable $W_{free}^I$-injectives so this proves Step 1.\\
            \textbf{Step 2:} $T_{free}^I\in \hat{\mathcal T}$:\\
            The torsion class $\mathcal T\vee\hat{\mathcal T}$ consists of those objects in $\modulecat \Lambda$ that admit a filtration by objects in $\mathcal T\cup\hat{\mathcal T}$. Let $0 = X_0\subset X_1\subset\dots\subset X_m = T_{free}^I$ be such a filtration. Since $\mathcal T\cup\hat{\mathcal T}\subset \mathcal T \vee \Gen W_{free}^I$ all of the $X_i$ are elements of $\mathcal T \vee \Gen W_{free}^I$. Further $W_{free}^I\subset\mathfrak a(\mathcal T \vee \Gen W_{free}^I)$. This implies that $X_i\in \mathfrak a(\mathcal T \vee \Gen W_{free}^I)$ for all $i$ because $X_i$ is a submodule of $T_{free}^I$ and \cref{frakaClosure}. In fact $X_i\in W_{free}^I$ because the simple objects of $W_{free}^I$ are a subset of the simple objects of $\mathfrak a(\mathcal T \vee \Gen W_{free}^I)$ and $T^I_{free}\in W^I_{free}$. Finally this shows that $0 = X_0\subset X_1\subset\dots\subset X_m = T_{free}^I$ is a filtration in $W_{free}^I$. In particular $X_i/X_{i-1}\notin \mathcal T$ for all $i$. But then each of those subfactors must be in $\hat{\mathcal T}$ and hence $T_{free}^I\in \hat{\mathcal T}$ by extension closure.
        \end{proof}

\section{Interval mutation and almost 2-cluster tilting objects}
    In this section we will interpret mutation of 2-cluster tilting objects in terms of mutable intervals. Any almost 2-cluster tilting object admits three completions to 2-cluster tilting objects. The mutable intervals corresponding to those completions form an interval mutation which we introduce next.
    \begin{Definition}
        A triple $(B,I,A)$ of mutable intervals is called an \emph{interval mutation} if
        \begin{enumerate}
            \item $I = A\sqcup B$
            \item $\max B = \max I$
            \item $\min A = \min I $
        \end{enumerate}
    \end{Definition}
    Interval mutations relate a bigger interval to two smaller ones. This forms the basis for an induction argument for the main theorem. We study interval mutations on their own first, then relate them to mutation of 2-cluster tilting objects at the end of this section.\\   
    To construct interval mutations we need to associate another wide subcategory to a mutable interval
    \begin{Definition}
        Let $ I = [(\mathcal{T},\mathcal{F}) \leq (\mathcal{T}',\mathcal{F}')]$ be a mutable interval. The \emph{$\delta$-sequence} of $I$ is the sequence of wide subcategories $(\mathfrak a (\mathcal T), \mathfrak a (\mathcal F)\cap \mathfrak a (\mathcal T'),\mathfrak a (\mathcal F'))$.
    \end{Definition}
    Note that any mutable interval can be recovered from its $\delta$-sequence.
    \begin{CatalandRemark}
        Under the identification of wide subcategories with noncrossing partitions this is the $\delta$-sequence of the 2-noncrossing partition corresponding to $I$.
    \end{CatalandRemark}
    To be able to use interval mutations for an inductive argument we need to show that any proper mutable interval is the middle term of an interval mutation. We will construct such an interval mutation using simple modules of $\mathfrak a (\mathcal F)\cap \mathfrak a (\mathcal T')$. The next two results show that this category does not vanish for proper intervals.
    \begin{Lemma}
        Let $ I = [(\mathcal{T},\mathcal{F}) \leq (\mathcal{T}',\mathcal{F}')]$ be a mutable interval.
        Then we have 
        \[\rank(\mathfrak a (\mathcal T)) + \rank(\mathfrak a (\mathcal F)\cap \mathfrak a (\mathcal T')) +\rank(\mathfrak a (\mathcal F')) = |\Lambda|\]
    \end{Lemma}
    \begin{proof}
        We have
        \begin{align*}
            |\Lambda| &= \rank(\mathfrak a (\mathcal T)) + \rank(\mathfrak a (\mathcal F))\\
            &= \rank(\mathfrak a (\mathcal T)) + \rank(\mathfrak a (\mathcal F)\cap {}^{\perp_{0,1}}(\mathfrak a (\mathcal F'))) +\rank(\mathfrak a (\mathcal F'))\\
            &= \rank(\mathfrak a (\mathcal T)) + \rank(\mathfrak a (\mathcal F)\cap \mathfrak a (\mathcal T')) +\rank(\mathfrak a (\mathcal F')).
        \end{align*}
        Here the first and second equalities are \cref{wideRankComplement}.
        In addition the second equality uses $\mathfrak a (\mathcal F')\subset \mathfrak a (\mathcal F)$
    \end{proof}
    \begin{Corollary}
        Let $ I = [(\mathcal{T},\mathcal{F}) \lneq (\mathcal{T}',\mathcal{F}')]$ be a proper mutable interval.
        Then $\mathfrak{a}(\mathcal{F})\cap\mathfrak{a}(\mathcal{T}')$ is not zero.
    \end{Corollary}
    \begin{proof}
        If  $\mathfrak{a}(\mathcal{F})\cap\mathfrak{a}(\mathcal{T}') = 0$ then $|\Lambda| = \rank(\mathfrak a (\mathcal T)) + \rank(\mathfrak a (\mathcal F'))$. 
        Since $\mathfrak a (\mathcal T)\subseteq \mathfrak a (\mathcal T')$ and $|\Lambda| = \rank(\mathfrak a (\mathcal T')) + \rank(\mathfrak a (\mathcal F'))$ we must have
        $\mathfrak a (T) = \mathfrak a (T')$ and hence $\mathcal T = \mathcal T'$.
    \end{proof}
    Choosing a simple module in $\mathfrak{a}(\mathcal{F})\cap\mathfrak{a}(\mathcal{T}')$ yields an interval mutation as follows.  
    \begin{Proposition}
        Let $ I = [(\mathcal{T},\mathcal{F}) < (\mathcal{T}',\mathcal{F}')]$ be a mutable interval.
        Let $X$ be a simple module of $\mathfrak{a}(\mathcal{F})\cap\mathfrak{a}(\mathcal{T}')$.
        We call the pair $(I,X)$ an \emph{augmented interval} and set
        \[ A = [(\mathcal{T},\mathcal{F}) \leq (\mathcal{T}',\mathcal{F}')\wedge ({}^{\perp_0}X,\Cogen X)]\]
        and
        \[B = [(\mathcal{T},\mathcal{F})\vee (\Gen X,X^{\perp_0}) \leq (\mathcal{T}',\mathcal{F}')].\]
        Then $(B,I,A)$ is an interval mutation.
        We call it the \emph{interval mutation induced by $(I,X)$}. In particular both intervals $A$ and $B$ are mutable.
    \end{Proposition}
    \begin{proof}
     We clearly have $\max B = \max I$ and $\min A = \min I$. Because the object $X$ is torsion for every torsion pair in $B$ and torsion-free for every torsion pair in $A$ we have $A\cap B = \emptyset$. It remains to check that $A$ and $B$ are mutable and that their union is $I$.
    \begin{Claim}
        The intervals $A$ and $B$ are mutable.
    \end{Claim}
    \begin{proof}
        We only show this for $B$, the argument for $A$ is dual. Consider the torsion class $\hat{\mathcal T}$ of $\mathfrak a(\mathcal T')$ given by  $\hat{\mathcal T} = (\mathcal T\cap\mathfrak a (\mathcal T'))\vee (\Gen X\cap \mathfrak a (\mathcal T'))$ where the join is taken in $\tors \mathfrak a(\mathcal T')$. Then $\hat{\mathcal T}$ is generated by a wide subcategory $W\subset \mathfrak a(\mathcal T')$. From $\Gen \hat{\mathcal T} = \mathcal T \vee \Gen X$ we get  $W = \mathfrak a (\mathcal T \vee \Gen X)$ and hence $A$ is mutable . 
    \end{proof}
    \begin{Claim}
        $I = A\sqcup B$
    \end{Claim}
    \begin{innerproof}
        Let $(\mathcal{T},\mathcal{F}) \leq(\hat{\mathcal{T}},\hat{\mathcal{F}})\leq (\mathcal{T}',\mathcal{F}')$ be a torsion pair in $I$.
        We consider the canonical exact sequence of $X$ with respect to $(\hat{\mathcal{T}},\hat{\mathcal{F}})$:
        \[0\rightarrow T\rightarrow X\rightarrow F\rightarrow 0\]
        From $T\in\hat{\mathcal{T}}\subseteq \mathcal{T}'$ and $X\in \mathfrak a(\mathcal{T}')$ we get $T\in \mathfrak a(\mathcal{T}')$ by \cref{frakaClosure}.
        The canonical exact sequence above then implies $F\in \mathfrak a(\mathcal{T}')$.
        Dually we must have $T\in \mathfrak a(\mathcal{F})$ and $F \in \mathfrak a(\mathcal{F})$.
        In particular the canonical exact sequence lives in $\mathfrak a(\mathcal{F})\cap \mathfrak a(\mathcal{T}')$. Because $X$ is simple in that category we must have $X\cong T$ or $X \cong F$. The former implies $(\hat{\mathcal{T}},\hat{\mathcal{F}})\in B$ and the latter $(\hat{\mathcal{T}},\hat{\mathcal{F}})\in A$.
    \end{innerproof}
    \end{proof}
    In fact all interval mutations arise in this way from augmented intervals. In this way the augmented intervals behave similar to almost 2-cluster tilting objects.
    \begin{Proposition}
        Let $(B,I,A)$ be an interval mutation. Then there exists $X\in \modulecat \Lambda$ such that $(I,X)$ is an augmented interval that induces $(B,I,A)$.
    \end{Proposition}
    \begin{proof}
        Let $I$ be the interval $ I = [(\mathcal{T},\mathcal{F}) < (\mathcal{T}',\mathcal{F}')]$ and set $\min B = (\mathcal{T}_B,\mathcal{F}_B)$ and $\max A = (\mathcal{T}_A,\mathcal{F}_A)$. 
        Because $\min B \gneq \min I$ there must be a torsion pair $(\hat{\mathcal T},\hat{\mathcal F}) \geq \min I$ covered by $\min B$. In particular $(\hat{\mathcal T},\hat{\mathcal F})\in A$. In fact this defines a unique torsion pair: If a second pair $(\Tilde{\mathcal T},\Tilde{\mathcal F})$ satisfies the same conditions, then $\min B = (\hat{\mathcal T},\hat{\mathcal F})\vee (\Tilde{\mathcal T},\Tilde{\mathcal F})$ and hence $\min B\in A$ which is impossible. Let $X$ be the edge label of the cover relation $(\hat{\mathcal T},\hat{\mathcal F})\lessdot \min B$.
        Dually we obtain a cover relation $\max A \lessdot (\Tilde{\mathcal T},\Tilde{\mathcal F})$. 
        We want to show that this second cover relation also has edge label $X$. By \cref{edgeLabelViaJoinIrred} we have $\Gen X\subset \mathcal{T}_B\subset \Tilde{\mathcal T}$ and $\Gen X\setminus \{X\}\subset \hat{\mathcal T}\subset \mathcal{T}_A$. In addition $X\notin \mathcal{T}_A$ because otherwise $\mathcal{T}_B = \hat{\mathcal T}\vee \Gen X \subset \mathcal{T}_A$. Again by  \cref{edgeLabelViaJoinIrred} the module $X$ must be the edge label of $\max A \lessdot (\Tilde{\mathcal T},\Tilde{\mathcal F})$.
        Now \cref{SimplesAreLabels} shows that $X\in \mathfrak a(\mathcal{T}_B)\subset \mathfrak a(\mathcal T')$ and $X\in \mathfrak a(\mathcal{F}_A)\subset \mathfrak a(\mathcal F)$ and hence $X\in \mathfrak a (\mathcal F)\cap \mathfrak a(\mathcal T')$. 
        We set \[ A' = [(\mathcal{T},\mathcal{F}) \leq (\mathcal{T}',\mathcal{F}')\wedge ({}^{\perp_0}X,\Cogen X)]\]
        and
        \[B' = [(\mathcal{T},\mathcal{F})\vee (\Gen X,X^{\perp_0}) \leq (\mathcal{T}',\mathcal{F}')].\]
        Because $X\in \mathcal T_B$ we have $B\subset B'$ and dually $A\subset A'$. Since $X$ is torsion for all torsion pairs in $B'$ and torsion-free for all pairs in $A'$ we get $A'\cap B' = \emptyset$ and hence $A= A'$ and $B=B'$.
        The only thing left to show is that $X$ is simple in $\mathfrak a (\mathcal F)\cap \mathfrak a(\mathcal T')$. We choose a short exact sequence
        \[0\rightarrow X''\rightarrow X\rightarrow X'\rightarrow 0\] in $\mathfrak a (\mathcal F)\cap \mathfrak a(\mathcal T')$ with $X'$ simple in $\mathfrak a (\mathcal F)\cap \mathfrak a(\mathcal T')$ and consider the torsion pair $(\mathcal T'\cap {}^{\perp_0}X',\mathcal 
        F' \wedge \Cogen X')\in I$. With respect to this torsion pair $X$ has a torsion free quotient: $X'$. Because $X$ is either torsion or torsion-free for every torsion pair in $I$, it must be torsion free here. Since $X''$ is a submodule of $X$, it must also be torsion-free. If $X''$ is not zero, then it contains an indecomposable direct summand $\Tilde X$ with $\Ext_\Lambda^1(X', \Tilde X) \neq 0$. Because $\Lambda$ is directed, we have $\Hom_\Lambda(\Tilde X, X') = 0$. In particular $\Tilde X \in \mathcal T'\cap {}^{\perp_0}X'$ is torsion. As a summand of $X''$ it is also torsion free. This contradiction shows that $X'' = 0$ and further $X \cong X'$ is simple in $\mathfrak a (\mathcal F)\cap \mathfrak a(\mathcal T')$.
    \end{proof}
     Before describing the bijection to almost 2-cluster tilting objects we count augmented intervals. This is the analogue of \cref{almostCounting}.
    \begin{Proposition}\label{augmentedCounting}
        Let $n$ be the rank of $\Lambda$. Then we have
        \[|\{\textnormal{augmented intervals over }\Lambda\}| = \frac{n}{3}|\{\textnormal{mutable intervals in }\tors \Lambda\}|\]
    \end{Proposition}
    \begin{proof}
        Let $I$ be a mutable interval with delta sequence $(W_1,W_2,W_3)$. Then the number of augmented intervals of the form $(I,X)$ is the rank of $W_2$. We consider the following permutation on the set of mutable intervals interpreted as delta sequences:
        \begin{align*}
            \krew_2\colon \{\delta\text{-sequences}\} &\rightarrow \{\delta\text{-sequences}\}\\
            (W_1,W_2,W_3)&\mapsto (W_2,W_3,W_2^{\perp_{0,1}}\cap W_2^{\perp_{0,1}}).
        \end{align*}
        Its inverse is given by 
        \begin{align*}
            \krew_2^{-1}\colon \{\delta\text{-sequences}\} &\rightarrow \{\delta\text{-sequences}\}\\
            (W_1,W_2,W_3)&\mapsto ({}^{\perp_{0,1}}W_1\cap {}^{\perp_{0,1}}W_2,W_1,W_2).
        \end{align*}
        In particular there are in total $\rank(W_1) + \rank(W_2)+\rank(W_3) = n$ augmented intervals of  one of the forms $(\krew_2^{-1}(I),X)$,$(I,X)$ or $(\krew_2(I),X)$. Summing over all mutable intervals and accounting for the fact that each augmented interval is counted thrice yields the result.
    \end{proof}
    \begin{CatalandRemark}
        As observed earlier mutable intervals correspond to 2-noncrossing partitions and hence are counted by the 2-W-Catalan numbers. The permutation used in the proof is the 2-Kreweras complement.
    \end{CatalandRemark}
    Finally we can identify augmented intervals and almost 2-cluster tilting objects. An explicit example in type $A_3$ is given in \cref{ExampleSection}.
    \begin{Proposition}\label{augmentedIsAlmost}
        Let $(I,X)$ be an augmented interval and $(B,I,A)$ the corresponding interval mutation. Then there exists a unique almost 2-cluster tilting object $T^{(I,X)}$ with completions $T^B,T^I$ and $T^A$ such that $\interval(T^A) = A$, $\interval(T^I) = I$ and $\interval(T^B) = B$.  In addition the cyclic order on the completions of $T^{(I,X)}$ is given by $T^B\prec T^I\prec T^A \prec T^B$.
    \end{Proposition}
    \begin{proof}
        We write
        \[I = [(\mathcal T, \mathcal F)\leq(\mathcal T',\mathcal F')]\]
        and recall
        \[ A = [(\mathcal{T},\mathcal{F}) \leq (\mathcal{T}',\mathcal{F}')\wedge ({}^{\perp_0}X,\Cogen X)]\]
        and
        \[B = [(\mathcal{T},\mathcal{F})\vee (\Gen X,X^{\perp_0}) \leq (\mathcal{T}',\mathcal{F}')].\]
        For each $Y\in\{B,I,A\}$ we let $T^Y$ be the corresponding 2-cluster tilting object and decompose it as $T^Y = T^Y_{free}\oplus T^Y_{tors}[1] \oplus T^Y_{supp}[2]$. We distinguish three cases which will turn out to correspond to whether the complement of $T^{(I,X)}$ in $T^I$ is a summand of $T^I_{free}$, $T^I_{tors}[1]$ or $T^I_{supp}[2]$.\\
        \textbf{Case 1:} $X\in W(T^I_{free})$\\
        In this case the complement of $T^{(I,X)}$ in $T^I$ will turn out to be a summand of $T^I_{free}$.\\
        \textbf{Step 1.1:} $T^B_{supp} \cong T^I_{supp} \cong T^A_{supp}$\\
        The assumption $X\in W(T^I_{free})$ implies $\Hom_\Lambda(T^I_{supp},X) = 0$ In particular we get \[\Hom_\Lambda(T^I_{supp},\Gen X\vee\Gen T^I_{tors}) = 0 = \Hom_\Lambda(T^I_{supp},\Cogen X\wedge\Cogen T^I_{free}).\] The claim follows.\\
        \textbf{Step 1.2:} $W(T^I_{free}) = W(T^A_{free})$\\
        The objects $T^I_{free}$ and $T^A_{free}$ are the $\mathfrak a(\mathcal F)$-support tilting modules corresponding to the torsion free class $\mathcal F' \cap \mathfrak a(\mathcal F)$, respectively $(\Cogen X\wedge \mathcal F')\cap \mathfrak a(\mathcal F)$. Then the inclusion $\mathcal F' \subseteq (\Cogen X\wedge \mathcal F')$ shows $W(T^I_{free}) \subseteq W(T^A_{free})$. Conversely $X\in  W(T^I_{free})$ shows $W(T^A_{free}) \subseteq W(T^I_{free})$.\\
        \textbf{Step 1.3:} $T^A_{tors} \cong T^I_{tors}$\\
        From $W_{free}^I\cup W_{tors}^I\subset T_{supp}^{I\perp_{0,1}}$ and $|T_{free}^I|+|T_{tors}^I|+|T^I_{supp}| = |\Lambda|$ and the analogous statement for $A$ we get
        \[W(T^A_{tors}) = T_{supp}^{A\perp_{0,1}}\cap{}^{\perp_{0,1}}W(T^A_{free}) = T_{supp}^{I\perp_{0,1}}\cap{}^{\perp_{0,1}}W(T^I_{free}) = W(T^I_{tors}).\]
        In particular $T^A_{tors}$ and $T^I_{tors}$ are  both given by the tilting module of $\mathcal T\cap W(T^I_{tors})$ in $ W(T^I_{tors})$, that is $T^A_{tors} \cong T^I_{tors}$.\\
        \textbf{Step 1.4:} $T^A_{free}$ is a mutation of $T^I_{free}$ as an $\mathfrak a (\mathcal F)$-support tilting module\\
        We consider the torsion pair $(\Tilde{\mathcal T'},\Tilde{\mathcal F'} ):=(\mathcal T'\cap \mathfrak a (\mathcal F),\mathcal F'\cap \mathfrak a (\mathcal F))$ in $\tors \mathfrak a (\mathcal F).$ The wide subcategories of $\mathfrak a (\mathcal F)$ corresponding to $(\Tilde{\mathcal T'},\Tilde{\mathcal F'} )$ are $\mathfrak a(\Tilde{\mathcal T'}) =\mathfrak a (\mathcal F)\cap \mathfrak a(\mathcal{T}')$ and $\mathfrak a(\Tilde{\mathcal F'}) =\mathfrak a (\mathcal F')$. We recall that by definition $X$ is a simple object in $\mathfrak a (\mathcal F)\cap \mathfrak a(\mathcal{T}')$. In particular $X$ is a downlabel of $(\Tilde{\mathcal T'},\Tilde{\mathcal F'} )$ in  $\tors \mathfrak a (\mathcal F)$. The torsion pair corresponding to that downlabel is $(\Tilde{\mathcal T'},\Tilde{\mathcal F'} )\wedge ({}^{\perp_0}X,\Cogen X)$. We conclude that $\Cogen T^A_{free}$ and $\Cogen T^I_{free}$ are adjacent in $\tors \mathfrak a (\mathcal F)$ and hence $T^A_{free}$ is a mutation of $T^I_{free}$ as an $\mathfrak a (\mathcal F)$-support tilting module.\\
        \textbf{Step 1.5:} description of the almost 2-cluster tilting object $T^{(I,X)}$\\
        From Step 1.2 we get that $T^I_{free}$ and $T^A_{free}$ have the same rank. By the previous step we can decompose $T^I_{free} \cong T^{(I,X)}_{free} \oplus T_A$ and $T^I_{free} \cong T^{(I,X)}_{free} \oplus T_I$ with $T_A$ and $T_I$ indecomposable.\\
        We set $T^{(I,X)}_{tors}= T^I_{tors}$ and $T^{(I,X)}_{supp} = T^I_{supp}$. Then the 2-cluster tilting objects $T^I$ and $T^A$ are both completions of the almost 2-cluster tilting object $T^{(I,X)}$.\\
        \textbf{Step 1.6:} $T^I\prec T^A$ among the completions of $T^{(I,X)}$\\
        We have $\Ext_{\Lambda}^1(T_A,T_I)\neq 0$ since $T^{(I,X)}_{free} \oplus T_A  \oplus T_I$ can not be rigid and $T_A$ is ext-injective in $\Cogen T^A_{free}\supset \Cogen T^I_{free}$. In particular $\Ext_{\mathcal C^2(\Lambda)}^1(T^A,T^I)\neq 0$ and hence $T^I\prec T^A$.\\
        \textbf{Step 1.7:} $T^B_{tors}$ is a mutation of $T^I_{tors}$ as an $\mathfrak a (\mathcal T')$-support tilting module\\
        We proceed dual to Step 1.4 by considering the torsion pair $(\Tilde{\mathcal T},\Tilde{\mathcal F} ):=(\mathcal T\cap \mathfrak a (\mathcal T'),\mathcal F\cap \mathfrak a (\mathcal T'))$ in $\mathfrak a (\mathcal T')$. The associated wide subcategories are $\mathfrak a (\Tilde{\mathcal T}) = \mathfrak a (\mathcal T)$ and $\mathfrak a (\Tilde{\mathcal F}) = \mathfrak a (\mathcal F)\cap \mathfrak a (\mathcal T')$. In particular $X$ is an uplabel of $(\Tilde{\mathcal T},\Tilde{\mathcal F} )$ in $\tors \mathfrak a (\mathcal T')$ and the associated torsion pair is $(\Tilde{\mathcal T},\Tilde{\mathcal F} )\vee (\Gen X, X^{\perp_0})$. We conclude that $\Gen T^B_{tors}$ and $\Gen T^B_{tors}$ are adjacent in $\tors \mathfrak a (\mathcal T')$ and hence $T^B_{tors}$ is a mutation of $T^I_{tors}$ as an $\mathfrak a (\mathcal T')$-support tilting module.\\
        \textbf{Step 1.8:} $X$ is simple in $\mathfrak a (\mathcal T \vee \Gen X)$\\
        From the proof of step 1.7 we get that $X$ is a downlabel of the torsion pair $(\Tilde{\mathcal T},\Tilde{\mathcal F} )\vee (\Gen X, X^{\perp_0})$ in $\tors \mathfrak a (\mathcal T')$. This implies $X\in \mathfrak a ((\mathcal T \vee \Gen X) \cap \mathfrak a (\mathcal T'))=\mathfrak a (\mathcal T \vee \Gen X)$.
        \textbf{Step 1.9:} description of $T^B_{tors}$\\
        From Step 1.7 and $X\in W(T^B_{tors})\setminus W(T^I_{tors})$ we obtain $|T^B_{tors}| = |T^I_{tors}|+1$ and hence can write $T^B_{tors} \cong T^I_{tors}\oplus T_B$ with $T_B$ indecomposable. 
        We let $\hat T_B$ be the projective cover of $X$ in $\mathfrak a (\mathcal T \vee \Gen X)$. Note that $\hat T_B$ is indecomposable since it was constructed as the projective cover of a simple object (see Step 1.7). Then $\hat T_B$ is a projective object in $\mathfrak a (\mathcal T \vee \Gen X)$ and  hence a summand of $T^B_{tors}$. Since $X\in W(T^I_{free}) = W(T^I_{tors})^{\perp_{0,1}}$ we have $\Hom_\Lambda(T^I_{tors},X) = 0$. In particular $\hat T_B$ cannot be a summand of $T^I_{tors}$ and hence $\hat T_B \cong T_B$.\\
        \textbf{Step 1.10:} $T^B_{free} \cong T^{(I,X)}_{free}$\\
        We set $\rad(T_B) = \ker(T_B\rightarrow X)$ as the radical of $T_B$ as an object in $\mathfrak a (\mathcal T \vee \Gen X)$. Because $T_B$ is projective in this hereditary subcategory, it's radical is a direct sum of other projective objects. All of those projectives are direct summands of $T^B_{tors}$. In fact they are summands of $T^I_{tors}$, because they are different from $T_B$. That is $\rad(T_B) \in \addcat(T^I_{tors})$.
        We apply $\Hom_\Lambda(-,T^I_{free})$ to the short exact sequence $0\rightarrow \rad(T_B)\rightarrow T_B\rightarrow X\rightarrow 0$ to obtain the following part of a long exact sequence:
        \[\Hom_\Lambda(\rad(T_B),T^I_{free})\rightarrow\Ext^1_\Lambda(X,T^I_{free})\rightarrow \Ext^1_\Lambda(T_B,T^I_{free})\rightarrow \Ext^1_\Lambda(\rad(T_B),T^I_{free}).\]
        The first and last term are trivial, because $\rad(T_B) \in \addcat(T^I_{tors})$. Then the two middle terms must be isomorphic. In particular restricting from $T_{free}^I$ to $T^{(I,X)}_{free}$ yields
        \[0 = \Ext^1_\Lambda(X,T^{(I,X)}_{free})\cong \Ext^1_\Lambda(T_B,T^{(I,X)}_{free}).\]
        This implies that $T^{(I,X)}_{free}\oplus T^I_{tors}[1]\oplus T_B[1]\oplus T^I_{supp}[2]$ is 2-rigid. Counting summands shows that it is the 2-cluster tilting object corresponding to $B$.\\
        \textbf{Case 2:} $X\in W(T^I_{tors})$\\
        This case is entirely dual to Case 1. This time the complement of $T^{(I,X)}$ in $T^I$ is a summand of $T^I_{tors}[1]$, the complement in $T_A$ is a summand of $T^A_{free}$ and the complement in $T^B$ is a summand of $T^B_{tors}[1]$.\\
        \textbf{Case 3:} $X\notin W(T^I_{free})$ and $X\notin W(T^I_{tors})$\\
        In this case the complement of $T^{(I,X)}$ in $T^I$ will turn out to be a summand of $T^I_{supp}[2]$.\\
        \textbf{Step 3.1:} $\Hom_\Lambda(T^I_{supp},X)\neq 0$\\
        Assume $\Hom_\Lambda(T^I_{supp},X)= 0$. Then $X\in {T^I_{supp}}^{\perp_{0,1}}$. Now $(\Gen W(T^I_{tors}),W(T^I_{tors})^{\perp_0})$ is a torsion pair in the interval $I$, hence it is in in either the interval $A$ or the interval $B$.  If the torsion pair is in $B$, then
        $X\in \Gen W(T^I_{tors})$. The category $W(T^I_{tors})$ is spanned by a support tilting module in  $\mathfrak a(\mathcal T')$ and hence closed under $\mathfrak a(\mathcal T')$-submodules. Since in addition $X\in \mathfrak a(\mathcal T')$ we get $X\in W(T^I_{tors})$, a contradiction. If on the other hand the torsion pair is in $A$, then $X\in W(T^I_{tors})^{\perp_0}\cap {T^I_{supp}}^{\perp_{0,1}} = \Cogen W(T^I_{free})$. Arguing dually we get $X\in W(T^I_{free})$, another contradiction.\\
        \textbf{Step 3.2:} description of the almost 2-cluster tilting object $T^{(I,X)}$\\
        We decompose $T^I_{supp} = T^{(I,X)}_{supp}\oplus T_I$ such that $T^{(I,X)}_{supp}$ is maximal with $\Hom_\Lambda(T^{(I,X)}_{supp},X) = 0$. Further we set $T^{(I,X)}_{tors} = T^I_{tors}$ and $T^{(I,X)}_{free} = T^I_{free}$. We will show later that $T_I$ is indecomposable and hence that $T^{(I,X)} = T^{(I,X)}_{free}\oplus T^{(I,X)}_{tors}[1]\oplus T^{(I,X)}_{supp}[2]$ is an almost 2-cluster tilting object with $T^I$ being one of its completions.\\
        \textbf{Step 3.3:} $T^B_{supp} = T^{(I,X)}_{supp}$\\
        Both are isomorphic to the maximal projective module $P$ with $\Hom_\Lambda(P,\mathcal T\cup\mathcal F'\cup\Gen X) = 0$.\\
        \textbf{Step 3.4:} $T^B_{tors}$ is a mutation of $T^I_{tors}$ as an $\mathfrak a (\mathcal T')$-support tilting module\\
        This can be shown exactly as in Step 1.7. We can again write $T^B_{tors} \cong T^I_{tors}\oplus T_B$. As in Step 1.9 we obtain a short exact sequence $0\rightarrow\rad(T_B)\rightarrow T_B\rightarrow X\rightarrow 0$ with $\rad(T_B)\in \addcat T^I_{tors}$.\\
        \textbf{Step 3.5:} $T^A_{free}$ is a mutation of $T^I_{free}$ as an $\mathfrak a (\mathcal F)$-support tilting module\\
        This is dual to the previous step. We can write $T^A_{free} \cong T^I_{free}\oplus T_A$.\\
        \textbf{Step 3.6:} $\Hom_\Lambda(T^B_{tors}, T^I_{free}) = 0 =\Ext^1_\Lambda(T^B_{tors}, T^I_{free})$\\
        We apply $\Hom_\Lambda(-,T^I_{free})$ to the short exact sequence from Step 3.4  to obtain the long exact sequence
        \begin{align*}
            \Hom_\Lambda(X,T^I_{free})&\rightarrow\Hom_\Lambda(T_B,T^I_{free})\rightarrow\Hom_\Lambda(\rad(T_B),T^I_{free})\rightarrow\\
            \Ext^1_\Lambda(X,T^I_{free})&\rightarrow \Ext^1_\Lambda(T_B,T^I_{free})\rightarrow \Ext^1_\Lambda(\rad(T_B),T^I_{free}).
        \end{align*}
        We have $\Hom_\Lambda(X,T^I_{free}) = 0$ because $X\in \mathcal T'$ and $T^I_{free}\in \mathcal F'$. Further we observe $\Hom_\Lambda(\rad(T_B),T^I_{free}) = 0 = \Ext^1_\Lambda(\rad(T_B),T^I_{free})$ since $\rad(T_B)\in \addcat T^I_{tors}$. Finally $\Ext^1_\Lambda(X,T^I_{free}) = 0$ because $X\in \Cogen T^A_{free}$ and $T^I_{free}$ is ext-injective in $\Cogen T^A_{free}$ because its a summand of $T^A_{free}$. The step follows from exactness.\\
        \textbf{Step 3.7:} $T^I_{free} \cong T^B_{free}$\\
        Let $\widetilde {T^I_{free}}$ be the support tilting module of the torsion-free class $\mathcal F'$. The dual of the claim in the proof of \cref{clustersAreMutableBijection} shows that $T^B_{free}$ consists of exactly those summands $T$ of $\widetilde {T^I_{free}}$ with $\Ext^1_\Lambda(T^B_{tors},T) = 0$. With the previous step we obtain that $T^I_{free}$ is a direct summand of $T^B_{free}$. Similarly $T^I_{free}$ consists of those summands $T$ of $\widetilde {T^I_{free}}$ with $\Ext^1_\Lambda(T^I_{tors},T) = 0$. Since this is a weaker condition than $\Ext^1_\Lambda(T^B_{tors},T) = 0$, we conclude $T^I_{free} \cong T^B_{free}$.\\
        \textbf{Step 3.8:} $T_I$ is indecomposable\\
        Because $|T^B_{free}| = |T^I_{free}|$ and $|T^B_{tors}| = |T^I_{tors}| + 1$, we must have $|T^B_{supp}| = |T^I_{supp}|-1 = |T^I_{supp}|-|T_I|$ and hence that $T_I$ is indecomposable. In particular the object $T^{(I,X)}$ described earlier is an almost 2-cluster tilting object with completions $T^B$, $T^I$ and by the dual arguments $T^A$.\\
        \textbf{Step 3.9:} $T^B\prec T^I$ among the completions of $T^{(I,X)}$\\
        We have $\Hom_\Lambda(T_I,T_B)\neq 0$ and hence $\Ext^1_{\mathcal C^2(\Lambda)}(T_I[2],T_B[1])\neq 0$.
    \end{proof}
    \begin{Corollary}
        There is a bijection 
        \[\{\textnormal{almost 2-cluster tilting objects in } \mathcal C_2(\Lambda)\}\rightarrow \{\textnormal{augmented intervals in } \tors(\Lambda)\}\]
    \end{Corollary}
    \begin{proof}
        By \cref{almostCounting} and \cref{augmentedCounting} the two sets have the same size. \cref{augmentedIsAlmost} gives an injective map from the set of augmented intervals to the set of almost 2-cluster tilting objects.
    \end{proof}
    Shifting an almost 2-cluster tilting object yields another almost 2-cluster tilting object. Applying the Serre permutation to each element of an interval mutation should therefore yield another interval mutation. The next Proposition makes this precise. It also collects some results about the ranks of the categories $W_{free}^?$ that will be needed in the proof of the fractionally Calabi-Yau property. 
    \begin{Proposition}\label{Rotation}
        Let $(B,I,A)$ be an interval mutation. Then
        $W_{free}^B\subseteq W_{free}^I \subseteq W_{free}^A$ and $\rank(W_{free}^B)+1=\rank(W_{free}^A)$.
        This yields two cases:
        \begin{enumerate}
            \item if $ W_{free}^I = W_{free}^A$, then $(\mathbb S I,\mathbb S A,\mathbb S B)$ is an interval mutation and
            \item if $W_{free}^B = W_{free}^I$, then $(\mathbb S A,\mathbb S B,\mathbb S I)$ is an interval mutation.
        \end{enumerate}
    \end{Proposition}
    \begin{proof}
        We let $T^{(I,X)}$ be the almost cluster tilting module corresponding to $(B,I,A)$ and let $T_B,T_I$ and $T_A$ be the three complements of $T^{(I,X)}$ labeled such that $\interval(T^{(I,X)}\oplus T_A) =A$ and similar for $I$ and $B$. We consider the three cases from the proof of \cref{augmentedIsAlmost} separately.\\
        \textbf{In Case 1} we have $T_B\in \modulecat\Lambda[1]$ and $T_I,T_A\in \modulecat \Lambda$. In addition we have $W_{free}^B\subset W_{free}^I = W_{free}^A$ and $\rank(W_{free}^B)+1=\rank(W_{free}^I)$. After applying the shift we have $T_B[1]\in \modulecat \Lambda\cup\proj\Lambda[2]$ and $T_I[1],T_A[1]\in \modulecat \Lambda[1]$. This distribution of the complements  of  the almost 2-cluster tilting object $T^{(I,X)}[1]$ is only possible in Case 2. In addition $\interval(T^{(I,X)}[1]\oplus T_B[1])$ must be the third entry in the interval mutation corresponding to $T^{(I,X)}$ because in Case 2 the complement corresponding to that entry is in $\modulecat \Lambda$. Because $[1]$ preserves the cyclic order on complements and the order in the interval mutation matches the order on complements, we get that the interval mutation corresponding to $T^{(I,X)}[1]$ must be
        \[(\interval(T^{(I,X)}[1]\oplus T_I[1]),\interval(T^{(I,X)}[1]\oplus T_A[1]),\interval(T^{(I,X)}[1]\oplus T_B[1])) = (\mathbb S I,\mathbb S A,\mathbb S B).\]
        \textbf{In Case 2} we have $T_B,T_I\in \modulecat\Lambda[1]$ and $T_A\in \modulecat \Lambda$. In addition we have $W_{free}^B= W_{free}^I \subset W_{free}^A$ and $\rank(W_{free}^I)+1=\rank(W_{free}^A)$. After applying the shift we have $T_B[1],T_I[1]\in \modulecat \Lambda\cup\proj\Lambda[2]$ and $T_A[1]\in \modulecat \Lambda[1]$. This is possible in both Case 1 and Case 3. Either way $\interval(T^{(I,X)}[1]\oplus T_A[1])$ must be the first entry in the interval mutation corresponding to $T^{(I,X)}$ and hence we get the interval mutation $(\mathbb S A,\mathbb S B,\mathbb S I)$.\\
        \textbf{In Case 3} we have $T_B\in \modulecat\Lambda[1]$, $T_A\in \modulecat \Lambda$. and $T_I\in \proj \Lambda[2]$. In addition we have $W_{free}^B= W_{free}^I \subset W_{free}^A$ and $\rank(W_{free}^I)+1=\rank(W_{free}^A)$. After applying the shift we have $T_B[1],T_I[1]\in \modulecat \Lambda\cup\proj\Lambda[2]$ and $T_A[1]\in \modulecat \Lambda[1]$. This is the same distribution we obtained by applying the shift to Case 2 and hence we again get the interval mutation $(\mathbb S A,\mathbb S B,\mathbb S I)$.
    \end{proof}
    
\section{The Serre functor and representations of incidence algebras}
    In this section we recall the definition and basic properties of incidence algebras and of the Serre functor. Then we show how to compute the Serre functor for certain modules of incidence algebras called boolean antichain modules.
    \begin{Definition}
        Let $\mathcal{C}$ be a $k$-linear and Hom-finite additive category, where $k$ is a field. A \emph{Serre functor} is an equivalence $\bS\colon \mathcal C\rightarrow \mathcal C$, together with natural isomorphisms
        \[\Hom_\mathcal{C}(X,Y)\cong \Hom_\mathcal{C}(Y,\bS X)^*\]
        for all objects $X,Y \in \mathcal{C}$.
        If $\mathcal C$ admits a Serre functor than it is unique up to natural isomorphism.
    \end{Definition}
    We refer for example to \cite[Chapter 6.3]{Kra} for more information on Serre functors. The bounded derived category of finite dimensional algebras of finite global dimension provides a large class of examples of categories with a Serre functor.
    \begin{Proposition} \cite[Theorem 6.4.13]{Kra}
        Let $R$ be a finite dimensional algebra with finite global dimension. Then $D^b(R)$ admits a Serre functor $\bS$ given by the derived Nakayama functor.
    \end{Proposition}
    For the rest of this section let $L$ be a finite lattice and $k$ a field. 
    \begin{Definition}
        The \emph{incidence algebra} $kL$ has as basis the set of pairs $\{(a,b)\in L^2 \mid a\leq b\}$. The multiplication is given by 
        \[(a,b) * (c,d) = \begin{cases}
        (a,d) &\text{if }b=c\\
        0 &\text{else}
        \end{cases}\]
    \end{Definition}
    \begin{Definition}
        A \emph{representation} $M = (M_{a\in L},M_{a< b\in L^2})$ of $L$ consists of a vector space $M_i$ for each $i\in L$ and a linear map $M_{a< b}\colon M_a\rightarrow M_b$ for each comparable pair $a \lneq b\in L^2$ such that for each chain $a< b< c \in L^3$ we have $M_{a< c} = M_{b< c}\circ M_{a< b}$. A \emph{homomorphism} $f = (f_{a\in L})\colon M\rightarrow N$ of representations of $L$ is a collection of linear maps $f_a\colon M_a\rightarrow N_a$ such that $f_b \circ M_{a< b} = N_{a< b}\circ f_a$ for each pair $a< b\in L^2$.
    \end{Definition}
    \begin{Warning}
        There is a different notion of a representation of a lattice which requires the maps $M_{a<b}$ to be injective. We do not require this here.
    \end{Warning}
    The category of representations of $L$ is equivalent to the category of right-modules $\modulecat kL$ and we will identify the two. The most important modules of $L$ for us are the interval modules which we introduce next.
    \begin{Definition}
        Given an interval $I = [a,b]\subset L$ we define the \emph{interval module} $M_I$ by 
        $M_{I,x} = k$ if $x \in I$ and $f_{x< y} = \identity_k$ if $x, y\in I$.
    \end{Definition}
    \begin{Proposition}
        The simple modules in $\modulecat kL$ are exactly the modules of the form $M_{[a,a]}$ for $a\in L$. The projectives are the $M_{[a,\max L]}$ and the injectives are the $M_{[\min L, a]}$ for $a\in L$.
    \end{Proposition}
    To simplify notation we write $P_a = M_{[a,\max L]}$, $I_a = M_{[\min L,a]}$ and $S_a = M_{[a,a]}$. Morphisms between interval modules are easy to describe:
    \begin{Proposition}
        Let $M_{[a,b]}$ and $M_{[c,d]}$ be two interval modules then
        \[\dim \Hom_{kL}(M_{[a,b]},M_{[c,d]}) = \begin{cases}
            1 & \textnormal{if } c\leq a\leq d \leq b\\
            0 &\textnormal{else}
        \end{cases}\]
        Let $a,b\in L$ then we get for the projectives and injectives:
        \begin{align*}
            \dim \Hom_{kL}(P_a,P_b) &= \begin{cases}
            1 & \textnormal{if } b\leq a\\
            0 &\textnormal{else}
        \end{cases}& \dim \Hom_{kL}(I_a,I_b) &= \begin{cases}
            1 & \textnormal{if } b\leq a\\
            0 &\textnormal{else}
        \end{cases}
        \end{align*}
    \end{Proposition}
    The incidence algebras of lattices always have finite global dimension and hence admit a Serre functor $\bS$. Because $\bS$ is given by the derived Nakayama functor we can assume without loss of generality that $\bS P_a = I_a$ and that for $b\leq a$ the canonical embedding $P_b\hookrightarrow P_a$ gets sent to the canonical projection $I_b\twoheadrightarrow I_a$. To compute the image of more general interval modules under the Serre functor we need to consider two bigger classes of modules: The antichain modules and the dual antichain modules. The intersection of these two classes recovers exactly the class of interval modules. We refer to \cite{Tal} for more information on antichain modules.
    \begin{Definition}
        A subset $C\subset L$ is called an \emph{antichain} if the elements of $C$ are pairwise incomparable. We speak of an antichain $C$ over $\alpha \in l$ if $\alpha$ is a strict lower bound for all elements of $C$. Dually an antichain $D$ is below $\beta$ if $\beta$ is a strict upper bound for all elements in $D$.  Given a, possibly empty, antichain $C$ over $\alpha$ the antichain module $M^C_\alpha$ is given by $M^C_{\alpha,x} = k$ if $x \in \{y\in L \mid \alpha \leq y,\ c \nleq y\ \forall c\in C\}$. Further given an antichain $D$ below $\beta$ the dual antichain module $M^\beta_D$ is given by $M^\beta_{D,x} = k$ if $x \in \{y\in L \mid \beta \geq y,\ d \ngeq y \ \forall d\in D\}$. In both cases the maps are given by $\identity_k$ wherever possible and $0$ else.
    \end{Definition}
    We can write an interval module $M_{[a,b]}$ as an antichain module as follows: Let $C = \min ([a,\max L]\setminus [a,b])$. Then $M_{[a,b]} = M^C_a$. The main advantage of antichain modules is that they admit a canonical projective resolution called the antichain resolution. They were first described in general in \cite{AntichainResolution}.
    \begin{Definition}
        Let $C$ be an antichain over $\alpha$ and $M^C_\alpha$ the corresponding antichain module. The \emph{antichain resolution} of $M^C_\alpha$ is a projective resolution of $M^C_\alpha$ of the form
        \[0\rightarrow P^{|C|}\rightarrow\dots\rightarrow P^i\rightarrow \dots\rightarrow P^1\rightarrow P_\alpha\rightarrow M^C_\alpha\rightarrow0\]
        where
        \[P^i = \bigoplus_{\substack{C'\subseteq C\\ |C'| = i}} P_{\vee C'}\]
        and the maps are given by the Koszul sign rule.
    \end{Definition}
    In particular every antichain module has a finite projective resolution. Since all simple modules are antichain modules we recover $\globaldim kL <\infty$. Dually we get an antichain coresolution for dual antichain modules:
    \begin{Definition}
        Let $D$ be an antichain under $\beta$ and $M^\beta_D$ the corresponding dual antichain module. The \emph{antichain coresolution} of $M^\beta_D$ is an injective resolution of $M^\beta_D$ of the form
        \[0\rightarrow M^\beta_D\rightarrow I_\beta\rightarrow I^1\rightarrow \dots\rightarrow I^i\rightarrow \dots\rightarrow I^{|D|}\rightarrow 0\]
        where
        \[I^i = \bigoplus_{\substack{D'\subseteq D\\ |D'| = i}} I_{\wedge D'}\]
        and the maps are given by the Koszul sign rule.
    \end{Definition}
    Applying the Serre functor to the antichain resolution of an antichain $C$ over $\alpha$ sometimes yields an antichain coresolution for a different antichain $D$ below $\beta$. In this case we obtain $\bS M^C_\alpha = M^\beta_D[|C|]$. The antichains for which this happens are called boolean, respectively dual boolean.\\
    \begin{Definition}
        Let $C$ be an antichain over $\alpha$. We consider the map of join semilattices
        \begin{align*}
            \gamma^C_\alpha\colon\mathcal P(C)&\rightarrow L\\
            \emptyset &\mapsto \alpha\\
            A &\mapsto \bigvee A \text{ for } A \neq \emptyset 
        \end{align*}
    where $\mathcal P(C)$ is the boolean lattice of subsets of $C$ ordered by inclusion. We call the antichain $C$ over $\alpha$ \emph{boolean} if $\gamma^C_\alpha$ is an embedding of lattices.\\
    Dually let $D$ be an antichain under $\beta$. We consider the map of meet semilattices
    \begin{align*}
            \gamma^\beta_D\colon\mathcal P(C)^{op}&\rightarrow L\\
            \emptyset &\mapsto \beta\\
            A &\mapsto \bigwedge A \text{ for } A \neq \emptyset 
        \end{align*}
    where $\mathcal P(C)^{op}$ is the boolean lattice of subsets of $C$ ordered by reverse inclusion. We call the antichain $D$ under $\beta$ \emph{dual boolean} if $\gamma^\beta_D$ is an embedding of lattices.
    \end{Definition}
    There are bijections
    \[\{\text{boolean } C \text{ over }\alpha\}\leftrightarrow \{\text{boolean sublattices of } L\}\leftrightarrow \{\text{dual boolean }D \text{ under } \beta\}\]
    where we associate to a boolean sublattice $B\subset L$ the antichain consisting of its atoms over $\min B$ and the antichain consisting of its coatoms under $\max B$. The following result is folklore. Special cases have appeared in \cite{RognerudTamari}, \cite{Tal} and \cite{AntichainResolution}, but we could not find the general statement in the literature.
    \begin{Proposition}\label{serreOfBoolean}
        Let $C$ be a boolean antichain over $\alpha$. Then $\bS M^C_\alpha =M^\beta_D[|C|]$, where $D = \{\bigvee C'\mid C'\subset C, |C'| = |C|-1\}$ under $\beta = \bigvee C$ is the corresponding dual boolean antichain.
    \end{Proposition}
    \begin{proof}
        We note that with this choice for $D$ and $\beta$ the images of the maps $\gamma^C_\alpha$ and $\gamma^\beta_D$ agree. We call the boolean sublattice given by the images of these maps $B$ and for $x\in B$ we use $|x|$ to refer to the rank of $x$ with respect to $B$. In $\mathcal D^b(kL)$ we use the antichain resolution to obtain 
        \[M^C_\alpha\cong (P^{|C|}\rightarrow\dots\rightarrow P^i\rightarrow \dots\rightarrow P^1\rightarrow P_\alpha)\]
        where
        \[P^i = \bigoplus_{\substack{C'\subseteq C\\ |C'| = i}} P_{\vee C'} = \bigoplus_{\substack{x\in B\\ |x| = i}} P_{x}.\]
        Applying the Serre functor to each $P^i$ we get 
        \[\bS P^i = \bigoplus_{\substack{x\in B\\ |x| = i}} \bS P_{x} = \bigoplus_{\substack{x\in B\\ |x| = i}} I_{x} = \bigoplus_{\substack{D'\subseteq D\\ |D'| = |C|-i}} I_{\wedge D'} = I^{|D'|-i}\]
        and hence using the antichain coresolution
        \[\bS M^C_\alpha \cong (I_\beta\rightarrow I^1\rightarrow \dots\rightarrow I^i\rightarrow \dots\rightarrow I^{|D|})\cong M^\beta_D[|C|].\]
    \end{proof}
\section{Cambrian lattices are periodic Serre formal}
    We now combine the results from the previous chapters to prove that the lattice $\tors \Lambda$ is periodic Serre formal and hence fractionally Calabi-Yau. The lattices of the form $\tors \Lambda$ belong to the class of Cambrian lattices and at the end of the section we show how to generalise this result to all Cambrian lattices. We start by introducing periodic Serre formal and fractionally Calabi-Yau algebras and lattices. 
    \begin{Definition}
        Let $R$ be a finite dimensional algebra with finite global dimension. We call $R$ \emph{Serre formal} if the orbits of the indecomposable injectives under the Serre functor consist of stalk modules. That is for each indecomposable injective $I$ and each $n\in \bZ$ there is a module $M\in \modulecat R$ and $k\in \bZ$ with $\bS^n I \cong M[k]$. Further we call $R$ \emph{periodic Serre formal} if there is $n>0,k\in\bZ$ with $\bS^n I \cong I[k]$ for each indecomposable injective $I$. 
        We call a finite lattice $L$ \emph{periodic Serre formal} if $kL$ is for every field $k$.
    \end{Definition}
    \begin{Definition}
        Let $R$ be a finite dimensional algebra with finite global dimension and $m,n\in \bN_{>0}$. The algebra $R$ is \emph{$(m,n)$-fractionally Calabi-Yau} if
        \[\bS^n\cong [m]\]
        in $\mathcal D^b(R)$.
        A finite lattice $L$ is \emph{$(m,n)$-fractionally Calab--Yau} if $kL$ is for every field $k$.
    \end{Definition}
    We do not know whether either property can depend on the field. For lattices the two properties were connected by Rognerud.
    \begin{Theorem}\cite[Theorem 1.2]{RognerudTamari}\label{formalisCY}
        Let $L$ be a periodic Serre formal lattice. We fix $m,n\in \bN_{>0}$ with $\bS^n I \cong I[m]$ for all indecomposable injectives $I$. Then $L$ is $(m,n)$-fractionally Calabi-Yau.
    \end{Theorem}
    In order to prove that $\tors \Lambda$ is periodic Serre formal, we will show that for each mutable interval $I$ there is a $k\in \bZ$ with $\bS M_I \cong M_{\bS I}[k]$. That is the Serre functor permutes the mutable interval modules up to shift. This implies periodic Serre formal because every indecomposable injective is of the form $M_I$ for a mutable interval $I$.
    \begin{Proposition}\label{TheyAreAllMutable}
        All indecomposable projective, injective and simple modules are of the form $M_I$ for certain mutable intervals $I$.
    \end{Proposition}
    \begin{proof}
        The simple objects correspond to the first family of mutable intervals in \cref{mutableExamples}. The projectives correspond to the second family, while the injectives correspond to the third family.
    \end{proof}
    \begin{Theorem}\label{CategoricalSerre}
        Let $I$ be a mutable interval in $\tors \Lambda$. Then we have
        \[\bS M_I \cong M_{\bS I}[k_I]\]
        where $k_I = \rank(W_{free}^I)$.
    \end{Theorem}
    \begin{proof}
        To simplify notation we write $k_J = \rank(W_{free}^J)$ for any mutable interval $J$.
        We proof the theorem by induction on the size of the interval.\\
        \textbf{Induction start:} Assume $I = [(\mathcal{T},\mathcal{F}) \leq (\mathcal{T},\mathcal{F})]$ is an interval consisting of just one element. Let $C = \{\mathcal T_1\dots \mathcal T_n\}$ be the torsion classes covering $\mathcal T$. Then we can describe the interval module $M_I$ as an antichain module by $M_I = M^C_\mathcal T$. We need to show that the antichain $C$ over $\mathcal T$ is boolean in order to compute $\bS M_I$.
        \begin{Claim}
            The antichain $C$ over $\mathcal T$ is boolean.
        \end{Claim}
        \begin{proof}
            We first note that the interval $[\mathcal T\leq \bigvee C]$ is atomic and hence wide by \cref{atomicIsWide}. Further by \cref{wideIntervalsAsTors} there is a representation finite hereditary algebra $\Lambda'$ with $\tors \Lambda' \cong [\mathcal T\leq \bigvee C]$ as lattices. We have to check that the map $\gamma^C_\mathcal T$ is an embedding of lattices. Because the image of $\gamma^C_\mathcal T$ is contained in $[\mathcal T\leq \bigvee C]$ we can restrict our considerations to this interval. By passing to $\tors \Lambda'$ we can assume without loss of generality that $\mathcal T =  \emptyset$ and $\mathcal T_i = \{S_i\}$ where the $S_i$ range through the simple objects in $\modulecat \Lambda$. Consider a subset $A\subset C$. Then $\gamma^C_\mathcal T(A) = \bigvee A$ is the torsion class consisting of all modules with support in the set $\{S_i\mid \mathcal \{S_i\}\in A\}$. In particular the map $\gamma^C_\mathcal T$ is injective. For two subsets $A,A'\subset C$ we have $\gamma^C_\mathcal T(A\cap A') = \gamma^C_\mathcal T(A) \wedge \gamma^C_\mathcal T(A')$ because the set of modules supported on the intersection of two sets is exactly the set of modules supported on both sets. This shows that $\gamma^C_\mathcal T$ is compatible with the meet and hence an embedding of lattices, that is that the antichain $C$ over $\mathcal T$ is boolean.
        \end{proof}
        Using \cref{SimplesAreLabels} we can write $\mathcal T_i = \mathcal T \vee \Gen S_i$ where the $S_i$ range over the simple objects of $\mathfrak a (\mathcal F)$. Because the antichain $C$ over $\mathcal T$ is boolean we obtain from \cref{serreOfBoolean} that $\bS M_I = \bS M^C_\mathcal T = M^\beta_D[|C|]$ where $\beta = \bigvee C = \mathcal T \bigvee_i \Gen S_i = \mathcal T \vee \Gen \mathfrak a(\mathcal F)$ and $D = \{\mathcal T\vee \bigvee_{j\neq i}\Gen S_j \mid S_i \text{ simple in }\mathfrak a (\mathcal F)\}$. From \cref{serrePermSimple} we get $M^\beta_D = M_{\bS I}$. Finally we have $|C| = \rank \mathfrak a (\mathcal F) = \rank(W^I_{free})$ where the first equality follows from $C$ being parametrized by simple in $\mathfrak a(\mathcal F)=$ and the second from $I$ being a simple interval. Now we can compute $\bS M_I = \bS M^C_\mathcal T = M^\beta_D[|C|] = M_{\bS I}[k_I]$.\\
        \textbf{Induction step:} Now assume $I=[(\mathcal{T},\mathcal{F}) < (\mathcal{T}',\mathcal{F}')]$ is a proper interval. We choose a module $X$ such that $(I,X)$ is an augmented interval with associated interval mutation $(B,I,A)$. The definition of an interval mutation yields a short exact sequence
        \[0\rightarrow M_B\rightarrow M_I\rightarrow M_A\rightarrow 0\]
        in $\modulecat \Lambda$ and hence a triangle
        \[M_B\rightarrow M_I\rightarrow M_A\rightarrow M_B[1]\]
        in $\mathcal D^b(k\tors\Lambda)$.  The intervals $A$ and $B$ are strictly contained in $I$ and hence we know $\bS M_B = M_{\bS B}[k_B]$ and $\bS M_A = M_{\bS A}[k_A]$ by induction. Applying the Serre functor and substituting these identities yields a second triangle
        \begin{equation}\label{SerreTriangle}
            M_{\bS B}[k_B]\rightarrow \bS M_I\rightarrow M_{\bS A}[k_A]\rightarrow M_{\bS B}[k_B+1].
        \end{equation}
        By \cref{Rotation} we know that either $k_I = k_B$ and $(\bS A,\bS B, \bS I)$ is an interval mutation or $k_I = k_A$ and $(\bS I, \bS A, \bS B)$ is an interval mutation. We consider the former case, the latter can be treated similarly.
        We again obtain a triangle
        \[M_{\bS A}\rightarrow M_{\bS B}\rightarrow M_{\bS I}\rightarrow M_{\bS A}[1].\]
        Shifting this triangle by $k_I$ and rotating once yields a triangle
        \begin{equation}\label{ShiftTriangle}
            M_{\bS B}[k_B]\rightarrow M_{\bS I}[k_I]\rightarrow M_{\bS A}[k_A]\rightarrow M_{\bS B}[k_B+1],
        \end{equation}
        where we have already used $k_I = k_B = k_A-1$. Finally since 
        \[\dim \Ext^1(M_{\bS A}[k_A],M_{\bS B}[k_B]) = \dim \Ext^1(\bS M_A,\bS M_B) = \dim \Ext^1(M_A,M_B) = 1\]
        and neither of the two triangles \eqref{SerreTriangle} and \eqref{ShiftTriangle} split, we must have $\bS M_I \cong M_{\bS I}[k]$. 
    \end{proof}
    \begin{Corollary}\label{crystallographicSerreFormal}
        The lattice $\tors \Lambda$ is periodic Serre formal.
    \end{Corollary}
    \begin{proof}
        By \cref{TheyAreAllMutable} all indecomposable injectives are of the form $M_I$ for some mutable interval $I$. Now iterated application of \cref{CategoricalSerre} shows that $k\tors\Lambda$
        is Serre formal.
    \end{proof}
     \begin{Theorem}\label{torsCalabiYau}
        Let $h$ be the Coxeter number of $\Lambda$ and $N$ the number of indecomposable representations of $\Lambda$.
        Then $\tors \Lambda$ is $(2N,2h+2)$-fractionally Calabi-Yau.\\
        Further if $\Lambda$ is of Dynkin type $A_1$, $B_n$, $C_n$, $D_n$ ($n$ even), $E_7$, $E_8$ or $G_2$ then $\tors \Lambda$ is $(N,h+1)$-fractionally Calabi-Yau.
    \end{Theorem}
    \begin{proof}
        Applying \cref{CategoricalSerre} to the results of \cref{serrecount} yields $\bS^{2h+2} M_I\cong M_I[2N]$ (respectively $\bS^{h+1} M_I\cong M_I[N]$) for every mutable interval $I$. By \cref{TheyAreAllMutable} the indecomposable injectives $I$ are among the $M_I$  and hence we can apply \cref{formalisCY} to complete the proof. 
    \end{proof}
    We note that the Calabi-Yau dimension only depends on the Dynkin type of $\Lambda$. In fact Ladkani has shown in \cite{Ladkani} that for $\Lambda$ and $\Lambda'$ two different algebras of the same Dynkin type we have $\mathcal D^b(k\tors \Lambda)\cong \mathcal D^b(k\tors \Lambda')$.\\
    From the perspective of Coxeter theory the lattices of torsion classes we have examined so far belong to the class of Cambrian lattices introduced by Reading in \cite{ReadingCambrian}. We refer to \cite{CambrianSurvey} for more information on Cambrian lattices and a proper definition. We will use a slightly more ad-hoc definition in order to avoid having to introduce a lot of notation.
    \begin{Definition}
        A \emph{Cambrian lattice} is a finite product of elements of the following families:
        \begin{enumerate}[label = (\arabic*)]
            \item lattices of torsion classes of finite dimensional, connected, representation finite, hereditary algebras,
            \item $
            \begin{tikzcd}[ampersand replacement=\&,cramped,sep=tiny]
            	\& \circ \\
            	\circ \&\& \circ \\
            	\& \circ
            	\arrow[from=2-1, to=1-2]
            	\arrow[from=2-3, to=1-2]
            	\arrow[from=3-2, to=2-1]
            	\arrow[from=3-2, to=2-3]
            \end{tikzcd},
            \begin{tikzcd}[ampersand replacement=\&,cramped,sep=tiny]
            	\& \circ \\
            	\circ \&\& \circ \\
            	\&\& \circ \\
            	\& \circ
            	\arrow[from=2-1, to=1-2]
            	\arrow[from=2-3, to=1-2]
            	\arrow[from=4-2, to=2-1]
            	\arrow[from=3-3, to=2-3]
            	\arrow[from=4-2, to=3-3]
            \end{tikzcd},
            \begin{tikzcd}[ampersand replacement=\&,cramped,column sep=tiny, row sep = 0.2em]
            	\& \circ \\
            	\circ \&\& \circ \\
            	\&\& \circ \\
            	\&\& \circ \\
            	\& \circ
            	\arrow[from=2-1, to=1-2]
            	\arrow[from=2-3, to=1-2]
            	\arrow[from=5-2, to=2-1]
            	\arrow[from=3-3, to=2-3]
            	\arrow[from=4-3, to=3-3]
            	\arrow[from=5-2, to=4-3]
            \end{tikzcd},\dots,$
            \item 12 specific lattices.
        \end{enumerate}
        The elements of the first family are associated with the Dynkin types. The second family corresponds to the Coxeter types $I_n$ and the third family corresponds to Coxeter types $H_3$ and $H_4$. 
    \end{Definition}
    In particular most Cambrian lattices appear as lattices of torsion classes and this allows us to extend the previous Theorem to all Cambrian lattices
    \begin{Theorem}
        Cambrian lattices are periodic Serre formal and hence fractionally Calabi-Yau.
    \end{Theorem}
    \begin{proof}
        By \cite[Proposition 5.4]{ProductsofSerreFormal}
        Serre formal algebras are closed under taking tensor products and hence Serre formal lattices are closed under taking products. Furthermore the Serre functor distributes over tensor products, hence also periodic Serre formal is closed under products. Therefore it suffices to consider irreducible lattices. Let $L$ be an irreducible Cambrian lattices. If $L$ is crystallographic then it is Serre formal by \cref{crystallographicSerreFormal}. If $L$ is of type $I(m)$ then the orbits of the injectives ujnder the Serre functor are easy to describe combinatorially. We omit this here. Otherwise $L$ is one of the 12 Cambrian lattices of type $H_3$ or $H_4$. We have verified each of them by computer.
    \end{proof}
    For the irreducible Cambrian lattices we can give the dimension.
    \begin{Proposition}
        Let $L$ be an irreducible Cambrian lattice of Coxeter type $W$. Let $h$ be the Coxeter number of $W$ and $N$ the number of hyperplanes of the type $W$ reflection group.
        Then $L$ is $(2N,2h+2)$-fractionally Calabi-Yau.\\
        Further if $W$ is of Coxeter type  $A_1$, $B_n$, $C_n$, $D_n$ ($n$ even), $E_7$, $E_8$, $H_3$, $H_4$ or $I(m)$ ($m$ even) then $L$ is $(N,h+1)$-fractionally Calabi-Yau.
    \end{Proposition}
    \begin{proof}
        For the Dynkin types this is \cref{torsCalabiYau}. The Cambrian lattices of type $I(m)$ are derived equivalent to path algebras of type $D_{m+2}$ by \cite[Remark 5.10]{Yildirim}. This path algebra is fractionally Calabi-Yau with the required dimension. The types $H_3$ and $H_4$ can be checked by computer.
    \end{proof}
    Even in the non-crystallographic types all objects in the orbits of injectives under the Serre functor are shifted interval modules. It might be possible to adapt the theory of mutable intervals and its connection to 2-clusters to this context. We have not attempted this.
\section{Example: \texorpdfstring{$A_3$}{A₃}}\label{ExampleSection}
        In this section we give a detailed example for type $A_3$ with linear orientation. In this case the category $\mathcal C_2(\Lambda)$ looks as follows:
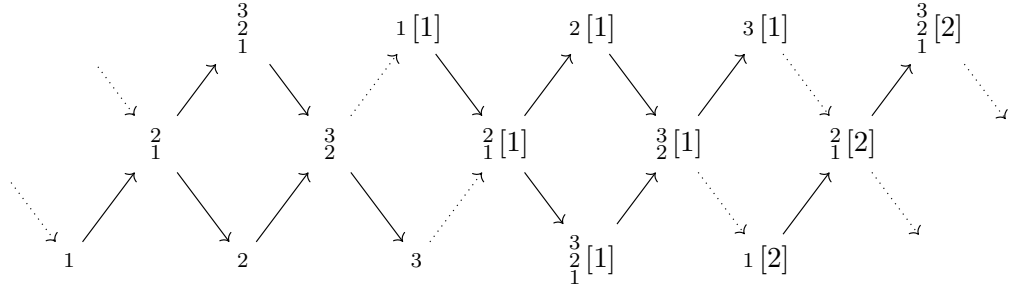
\[\begin{tikzcd}[ampersand replacement=\&, column sep = {3 em,between origins},row sep = {4 em,between origins}]
	\& {} \&\& \begin{array}{c} \begin{smallmatrix}3\\2\\1\end{smallmatrix} \end{array} \&\& {\begin{smallmatrix}1\end{smallmatrix}[1]} \&\& {\begin{smallmatrix}2\end{smallmatrix}[1]} \&\& {\begin{smallmatrix}3\end{smallmatrix}[1]} \&\& \begin{array}{c} \begin{smallmatrix}3\\2\\1\end{smallmatrix}[2] \end{array}\& \\
	{} \&\& \begin{array}{c} \begin{smallmatrix}2\\1\end{smallmatrix} \end{array} \&\& \begin{array}{c} \begin{smallmatrix}3\\2\end{smallmatrix} \end{array} \&\& \begin{array}{c} \begin{smallmatrix}2\\1\end{smallmatrix}[1] \end{array} \&\& \begin{array}{c} \begin{smallmatrix}3\\2\end{smallmatrix}[1] \end{array} \&\& \begin{array}{c} \begin{smallmatrix}2\\1\end{smallmatrix}[2] \end{array} \&\& {}\& \\
	\& {\begin{smallmatrix}1\end{smallmatrix}} \&\& {\begin{smallmatrix}2\end{smallmatrix}} \&\& {\begin{smallmatrix}3\end{smallmatrix}} \&\& \begin{array}{c} \begin{smallmatrix}3\\2\\1\end{smallmatrix}[1] \end{array} \&\& {\begin{smallmatrix}1\end{smallmatrix}[2]} \&\& {}\&
	\arrow[shorten <=10pt, dotted, from=1-2, to=2-3]
	\arrow[from=1-4, to=2-5]
	\arrow[from=1-6, to=2-7]
	\arrow[from=1-8, to=2-9]
	\arrow[dotted, from=1-10, to=2-11]
	\arrow[shorten >=10pt, dotted, from=1-12, to=2-13]
	\arrow[shorten <=10pt, dotted, from=2-1, to=3-2]
	\arrow[from=2-3, to=1-4]
	\arrow[from=2-3, to=3-4]
	\arrow[dotted, from=2-5, to=1-6]
	\arrow[from=2-5, to=3-6]
	\arrow[from=2-7, to=1-8]
	\arrow[from=2-7, to=3-8]
	\arrow[from=2-9, to=1-10]
	\arrow[dotted, from=2-9, to=3-10]
	\arrow[from=2-11, to=1-12]
	\arrow[shorten >=10pt, dotted, from=2-11, to=3-12]
	\arrow[from=3-2, to=2-3]
	\arrow[from=3-4, to=2-5]
	\arrow[dotted, from=3-6, to=2-7]
	\arrow[from=3-8, to=2-9]
	\arrow[from=3-10, to=2-11]
\end{tikzcd}\]
Here the outgoing arrows on the right are identified with the incoming arrows on the left. We use dashed arrows to visually separate the three subcategories $\modulecat \Lambda$, $\modulecat \Lambda[1]$ and $\proj \Lambda [2]$ from each other. Otherwise dashed and continuous arrows behave the same. In the next two diagrams we have marked a 2-cluster tilting object and its shift using double boxes. The cogenerated torsion-free class in $\modulecat \Lambda$ and the generated torsion-class in $\modulecat \Lambda [1]$ have been marked with single boxes.
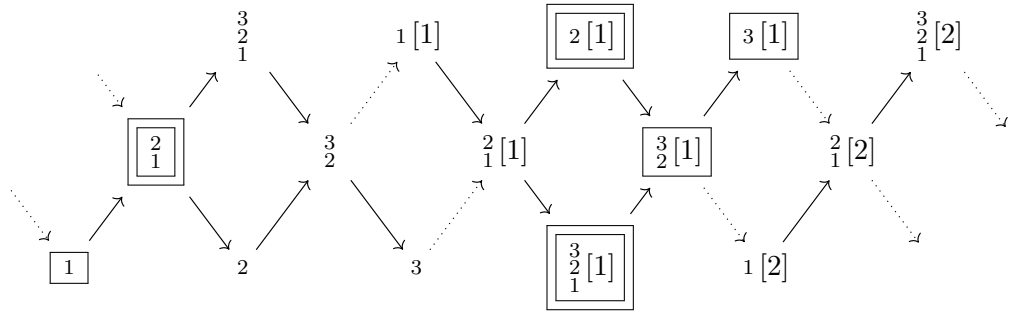
\[\begin{tikzcd}[ampersand replacement=\&, column sep = {3 em,between origins},row sep = {4 em,between origins}]
	\& {} \&\& \begin{array}{c} \begin{smallmatrix}3\\2\\1\end{smallmatrix} \end{array} \&\& {\begin{smallmatrix}1\end{smallmatrix}[1]} \&\& \fbox{\fbox{${\begin{smallmatrix}2\end{smallmatrix}[1]}$}} \&\& \fbox{${\begin{smallmatrix}3\end{smallmatrix}[1]}$} \&\& \begin{array}{c} \begin{smallmatrix}3\\2\\1\end{smallmatrix}[2] \end{array}\& \\
	{} \&\& \begin{array}{c} \fbox{\fbox{$\begin{smallmatrix}2\\1\end{smallmatrix}$}} \end{array} \&\& \begin{array}{c} \begin{smallmatrix}3\\2\end{smallmatrix} \end{array} \&\& \begin{array}{c} \begin{smallmatrix}2\\1\end{smallmatrix}[1] \end{array} \&\& \begin{array}{c} \fbox{$\begin{smallmatrix}3\\2\end{smallmatrix}[1]$} \end{array} \&\& \begin{array}{c} \begin{smallmatrix}2\\1\end{smallmatrix}[2] \end{array} \&\& {}\& \\
	\& \fbox{${\begin{smallmatrix}1\end{smallmatrix}}$} \&\& {\begin{smallmatrix}2\end{smallmatrix}} \&\& {\begin{smallmatrix}3\end{smallmatrix}} \&\& \begin{array}{c} \fbox{\fbox{$\begin{smallmatrix}3\\2\\1\end{smallmatrix}[1]$}} \end{array} \&\& {\begin{smallmatrix}1\end{smallmatrix}[2]} \&\& {}\&
	\arrow[shorten <=10pt, dotted, from=1-2, to=2-3]
	\arrow[from=1-4, to=2-5]
	\arrow[from=1-6, to=2-7]
	\arrow[from=1-8, to=2-9]
	\arrow[dotted, from=1-10, to=2-11]
	\arrow[shorten >=10pt, dotted, from=1-12, to=2-13]
	\arrow[shorten <=10pt, dotted, from=2-1, to=3-2]
	\arrow[from=2-3, to=1-4]
	\arrow[from=2-3, to=3-4]
	\arrow[dotted, from=2-5, to=1-6]
	\arrow[from=2-5, to=3-6]
	\arrow[from=2-7, to=1-8]
	\arrow[from=2-7, to=3-8]
	\arrow[from=2-9, to=1-10]
	\arrow[dotted, from=2-9, to=3-10]
	\arrow[from=2-11, to=1-12]
	\arrow[shorten >=10pt, dotted, from=2-11, to=3-12]
	\arrow[from=3-2, to=2-3]
	\arrow[from=3-4, to=2-5]
	\arrow[dotted, from=3-6, to=2-7]
	\arrow[from=3-8, to=2-9]
	\arrow[from=3-10, to=2-11]
\end{tikzcd}\]
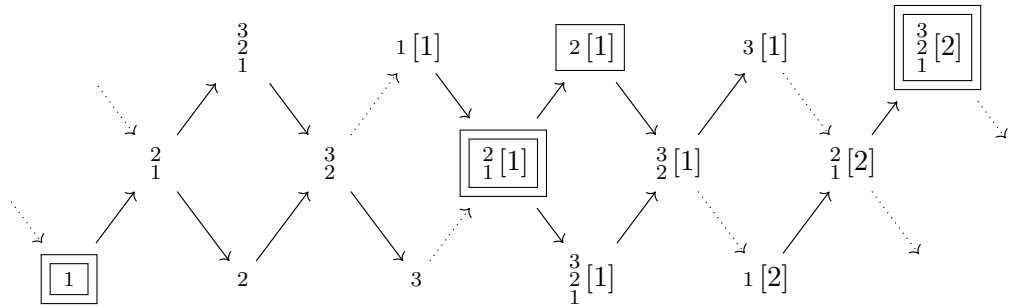
\[\begin{tikzcd}[ampersand replacement=\&, column sep = {3 em,between origins},row sep = {4 em,between origins}]
	\& {} \&\& \begin{array}{c} \begin{smallmatrix}3\\2\\1\end{smallmatrix} \end{array} \&\& {\begin{smallmatrix}1\end{smallmatrix}[1]} \&\& \fbox{${\begin{smallmatrix}2\end{smallmatrix}[1]}$} \&\& {\begin{smallmatrix}3\end{smallmatrix}[1]} \&\& \begin{array}{c} \fbox{\fbox{$\begin{smallmatrix}3\\2\\1\end{smallmatrix}[2]$}} \end{array}\& \\
	{} \&\& \begin{array}{c} \begin{smallmatrix}2\\1\end{smallmatrix} \end{array} \&\& \begin{array}{c} \begin{smallmatrix}3\\2\end{smallmatrix} \end{array} \&\& \begin{array}{c} \fbox{\fbox{$\begin{smallmatrix}2\\1\end{smallmatrix}[1]$}} \end{array} \&\& \begin{array}{c} \begin{smallmatrix}3\\2\end{smallmatrix}[1] \end{array} \&\& \begin{array}{c} \begin{smallmatrix}2\\1\end{smallmatrix}[2] \end{array} \&\& {}\& \\
	\& \fbox{\fbox{${\begin{smallmatrix}1\end{smallmatrix}}$}} \&\& {\begin{smallmatrix}2\end{smallmatrix}} \&\& {\begin{smallmatrix}3\end{smallmatrix}} \&\& \begin{array}{c} \begin{smallmatrix}3\\2\\1\end{smallmatrix}[1] \end{array} \&\& {\begin{smallmatrix}1\end{smallmatrix}[2]} \&\& {}\&
	\arrow[shorten <=10pt, dotted, from=1-2, to=2-3]
	\arrow[from=1-4, to=2-5]
	\arrow[from=1-6, to=2-7]
	\arrow[from=1-8, to=2-9]
	\arrow[dotted, from=1-10, to=2-11]
	\arrow[shorten >=10pt, dotted, from=1-12, to=2-13]
	\arrow[shorten <=10pt, dotted, from=2-1, to=3-2]
	\arrow[from=2-3, to=1-4]
	\arrow[from=2-3, to=3-4]
	\arrow[dotted, from=2-5, to=1-6]
	\arrow[from=2-5, to=3-6]
	\arrow[from=2-7, to=1-8]
	\arrow[from=2-7, to=3-8]
	\arrow[from=2-9, to=1-10]
	\arrow[dotted, from=2-9, to=3-10]
	\arrow[from=2-11, to=1-12]
	\arrow[shorten >=10pt, dotted, from=2-11, to=3-12]
	\arrow[from=3-2, to=2-3]
	\arrow[from=3-4, to=2-5]
	\arrow[dotted, from=3-6, to=2-7]
	\arrow[from=3-8, to=2-9]
	\arrow[from=3-10, to=2-11]
\end{tikzcd}\]
In the next diagram we have marked an almost 2-cluster tilting object with double boxes and its three complements using single boxes.
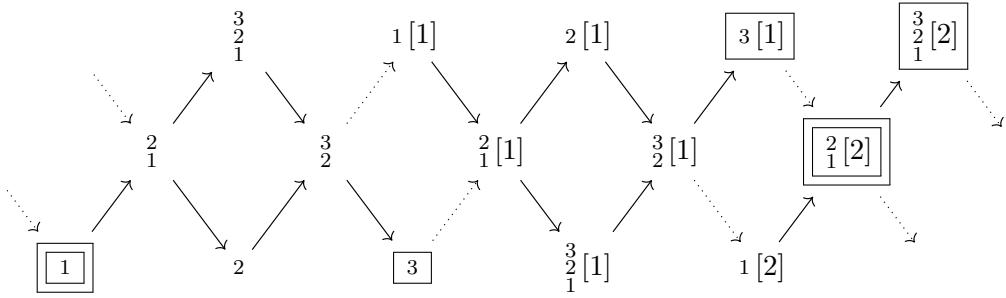
\[\begin{tikzcd}[ampersand replacement=\&, column sep = {3 em,between origins},row sep = {4 em,between origins}]
	\& {} \&\& \begin{array}{c} \begin{smallmatrix}3\\2\\1\end{smallmatrix} \end{array} \&\& {\begin{smallmatrix}1\end{smallmatrix}[1]} \&\& {\begin{smallmatrix}2\end{smallmatrix}[1]} \&\& \fbox{$\begin{smallmatrix}3\end{smallmatrix}[1]$} \&\& \begin{array}{c} \fbox{$\begin{smallmatrix}3\\2\\1\end{smallmatrix}[2]$} \end{array}\& \\
	{} \&\& \begin{array}{c} \begin{smallmatrix}2\\1\end{smallmatrix} \end{array} \&\& \begin{array}{c} \begin{smallmatrix}3\\2\end{smallmatrix} \end{array} \&\& \begin{array}{c} \begin{smallmatrix}2\\1\end{smallmatrix}[1] \end{array} \&\& \begin{array}{c} \begin{smallmatrix}3\\2\end{smallmatrix}[1] \end{array} \&\& \begin{array}{c} \fbox{\fbox{$\begin{smallmatrix}2\\1\end{smallmatrix}[2]$}} \end{array} \&\& {}\& \\
	\& \fbox{\fbox{$\begin{smallmatrix}1\end{smallmatrix}$}} \&\& \begin{smallmatrix}2\end{smallmatrix} \&\& \fbox{$\begin{smallmatrix}3\end{smallmatrix}$} \&\& \begin{array}{c} \begin{smallmatrix}3\\2\\1\end{smallmatrix}[1] \end{array} \&\& {\begin{smallmatrix}1\end{smallmatrix}[2]} \&\& {}\&
	\arrow[shorten <=10pt, dotted, from=1-2, to=2-3]
	\arrow[from=1-4, to=2-5]
	\arrow[from=1-6, to=2-7]
	\arrow[from=1-8, to=2-9]
	\arrow[dotted, from=1-10, to=2-11]
	\arrow[shorten >=10pt, dotted, from=1-12, to=2-13]
	\arrow[shorten <=10pt, dotted, from=2-1, to=3-2]
	\arrow[from=2-3, to=1-4]
	\arrow[from=2-3, to=3-4]
	\arrow[dotted, from=2-5, to=1-6]
	\arrow[from=2-5, to=3-6]
	\arrow[from=2-7, to=1-8]
	\arrow[from=2-7, to=3-8]
	\arrow[from=2-9, to=1-10]
	\arrow[dotted, from=2-9, to=3-10]
	\arrow[from=2-11, to=1-12]
	\arrow[shorten >=10pt, dotted, from=2-11, to=3-12]
	\arrow[from=3-2, to=2-3]
	\arrow[from=3-4, to=2-5]
	\arrow[dotted, from=3-6, to=2-7]
	\arrow[from=3-8, to=2-9]
	\arrow[from=3-10, to=2-11]
\end{tikzcd}\]

The lattice of torsion pairs looks as follows in this case:\\

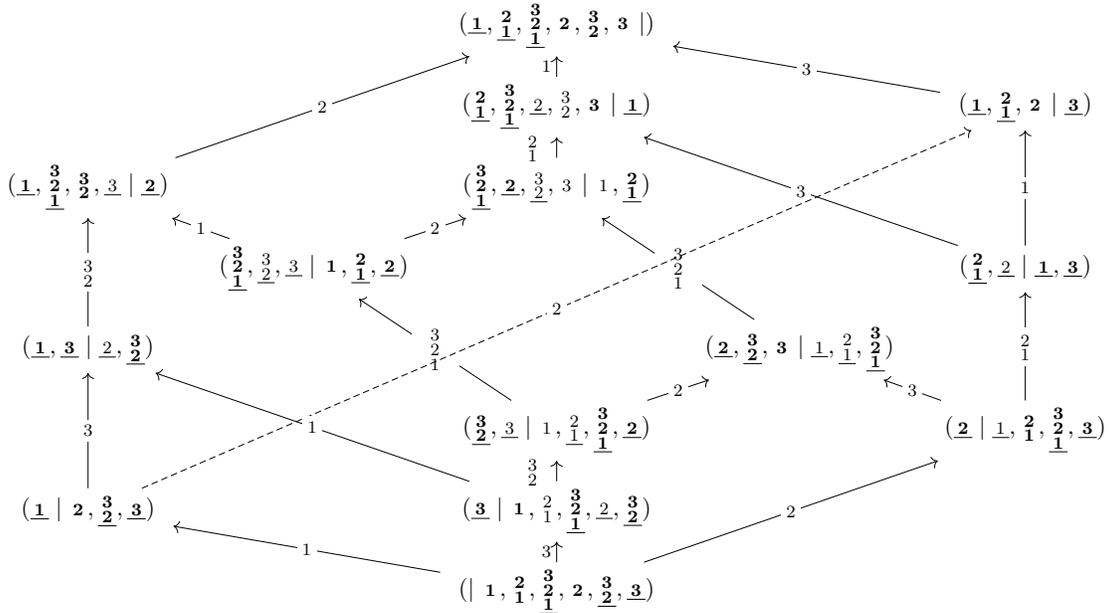
\adjustbox{max width=\textwidth}{    
$\begin{tikzcd}[ampersand replacement=\&,cramped,column sep=tiny,row sep=small]
	\&\& \begin{array}{c} (\underline{\begin{smallmatrix}\mathbf1\end{smallmatrix}},\underline{\begin{smallmatrix}\mathbf2\\\mathbf1\end{smallmatrix}},\underline{\begin{smallmatrix}\mathbf3\\\mathbf2\\\mathbf1\end{smallmatrix}},\begin{smallmatrix}\mathbf2\end{smallmatrix},\begin{smallmatrix}\mathbf3\\\mathbf2\end{smallmatrix},\begin{smallmatrix}\mathbf3\end{smallmatrix}\mid) \end{array} \\
	\&\& \begin{array}{c} (\underline{\begin{smallmatrix}\mathbf2\\\mathbf1\end{smallmatrix}},\underline{\begin{smallmatrix}\mathbf3\\\mathbf2\\\mathbf1\end{smallmatrix}},\underline{\begin{smallmatrix}2\end{smallmatrix}},\begin{smallmatrix}3\\2\end{smallmatrix},\begin{smallmatrix}\mathbf3\end{smallmatrix}\mid \underline{\begin{smallmatrix}\mathbf1\end{smallmatrix}}) \end{array} \&\& \begin{array}{c} (\underline{\begin{smallmatrix}\mathbf1\end{smallmatrix}},\underline{\begin{smallmatrix}\mathbf2\\\mathbf1\end{smallmatrix}},\begin{smallmatrix}\mathbf2\end{smallmatrix}\mid \underline{\begin{smallmatrix}\mathbf3\end{smallmatrix}}) \end{array} \\
	\begin{array}{c} (\underline{\begin{smallmatrix}\mathbf1\end{smallmatrix}},\underline{\begin{smallmatrix}\mathbf3\\\mathbf2\\\mathbf1\end{smallmatrix}},\begin{smallmatrix}\mathbf3\\\mathbf2\end{smallmatrix},\underline{\begin{smallmatrix}3\end{smallmatrix}}\mid \underline{\begin{smallmatrix}\mathbf2\end{smallmatrix}}) \end{array} \&\& \begin{array}{c} (\underline{\begin{smallmatrix}\mathbf3\\\mathbf2\\\mathbf1\end{smallmatrix}},\underline{\begin{smallmatrix}\mathbf2\end{smallmatrix}},\underline{\begin{smallmatrix}3\\2\end{smallmatrix}},\begin{smallmatrix}3\end{smallmatrix}\mid \begin{smallmatrix}1\end{smallmatrix},\underline{\begin{smallmatrix}\mathbf2\\\mathbf1\end{smallmatrix}}) \end{array} \\
	\& \begin{array}{c} (\underline{\begin{smallmatrix}\mathbf3\\\mathbf2\\\mathbf1\end{smallmatrix}},\underline{\begin{smallmatrix}3\\2\end{smallmatrix}},\underline{\begin{smallmatrix}3\end{smallmatrix}}\mid \begin{smallmatrix}\mathbf1\end{smallmatrix},\underline{\begin{smallmatrix}\mathbf2\\\mathbf1\end{smallmatrix}},\underline{\begin{smallmatrix}\mathbf2\end{smallmatrix}}) \end{array} \&\&\& \begin{array}{c} (\underline{\begin{smallmatrix}\mathbf2\\\mathbf1\end{smallmatrix}},\underline{\begin{smallmatrix}2\end{smallmatrix}}\mid \underline{\begin{smallmatrix}\mathbf1\end{smallmatrix}},\underline{\begin{smallmatrix}\mathbf3\end{smallmatrix}}) \end{array} \\
	\begin{array}{c} (\underline{\begin{smallmatrix}\mathbf1\end{smallmatrix}},\underline{\begin{smallmatrix}\mathbf3\end{smallmatrix}}\mid\underline{\begin{smallmatrix}2\end{smallmatrix}},\underline{\begin{smallmatrix}\mathbf3\\\mathbf2\end{smallmatrix}}) \end{array} \&\&\& \begin{array}{c} (\underline{\begin{smallmatrix}\mathbf2\end{smallmatrix}},\underline{\begin{smallmatrix}\mathbf3\\\mathbf2\end{smallmatrix}},\begin{smallmatrix}\mathbf3\end{smallmatrix}\mid\underline{\begin{smallmatrix}1\end{smallmatrix}},\underline{\begin{smallmatrix}2\\1\end{smallmatrix}},\underline{\begin{smallmatrix}\mathbf3\\\mathbf2\\\mathbf1\end{smallmatrix}}) \end{array} \\
	\&\& \begin{array}{c} (\underline{\begin{smallmatrix}\mathbf3\\\mathbf2\end{smallmatrix}},\underline{\begin{smallmatrix}3\end{smallmatrix}}\mid \begin{smallmatrix}1\end{smallmatrix},\underline{\begin{smallmatrix}2\\1\end{smallmatrix}},\underline{\begin{smallmatrix}\mathbf3\\\mathbf2\\\mathbf1\end{smallmatrix}},\underline{\begin{smallmatrix}\mathbf2\end{smallmatrix}}) \end{array} \&\& \begin{array}{c} (\underline{\begin{smallmatrix}\mathbf2\end{smallmatrix}}\mid\underline{\begin{smallmatrix}1\end{smallmatrix}},\begin{smallmatrix}\mathbf2\\\mathbf1\end{smallmatrix},\underline{\begin{smallmatrix}\mathbf3\\\mathbf2\\\mathbf1\end{smallmatrix}},\underline{\begin{smallmatrix}\mathbf3\end{smallmatrix}}) \end{array} \\
	\begin{array}{c} (\underline{\begin{smallmatrix}\mathbf1\end{smallmatrix}}\mid\begin{smallmatrix}\mathbf2\end{smallmatrix},\underline{\begin{smallmatrix}\mathbf3\\\mathbf2\end{smallmatrix}},\underline{\begin{smallmatrix}\mathbf3\end{smallmatrix}}) \end{array} \&\& \begin{array}{c} (\underline{\begin{smallmatrix}\mathbf3\end{smallmatrix}}\mid\begin{smallmatrix}\mathbf1\end{smallmatrix},\begin{smallmatrix}2\\1\end{smallmatrix},\underline{\begin{smallmatrix}\mathbf3\\\mathbf2\\\mathbf1\end{smallmatrix}},\underline{\begin{smallmatrix}2\end{smallmatrix}},\underline{\begin{smallmatrix}\mathbf3\\\mathbf2\end{smallmatrix}}) \end{array} \\
	\&\& \begin{array}{c} (\mid\begin{smallmatrix}\mathbf1\end{smallmatrix},\begin{smallmatrix}\mathbf2\\\mathbf1\end{smallmatrix},\underline{\begin{smallmatrix}\mathbf3\\\mathbf2\\\mathbf1\end{smallmatrix}},\begin{smallmatrix}\mathbf2\end{smallmatrix},\underline{\begin{smallmatrix}\mathbf3\\\mathbf2\end{smallmatrix}},\underline{\begin{smallmatrix}\mathbf3\end{smallmatrix}}) \end{array}
	\arrow["1", from=2-3, to=1-3]
	\arrow["3"{description}, from=2-5, to=1-3]
	\arrow["2"{description}, from=3-1, to=1-3]
	\arrow["\begin{array}{c} \begin{smallmatrix}2\\1\end{smallmatrix} \end{array}", from=3-3, to=2-3]
	\arrow["1"{description}, from=4-2, to=3-1]
	\arrow["2"{description}, from=4-2, to=3-3]
	\arrow["3"{description}, from=4-5, to=2-3]
	\arrow["1"{description}, from=4-5, to=2-5]
	\arrow["\begin{array}{c} \begin{smallmatrix}3\\2\end{smallmatrix} \end{array}"{description}, from=5-1, to=3-1]
	\arrow["\begin{array}{c} \begin{smallmatrix}3\\2\\1\end{smallmatrix} \end{array}"{description}, from=5-4, to=3-3]
	\arrow["\begin{array}{c} \begin{smallmatrix}3\\2\\1\end{smallmatrix} \end{array}"{description}, from=6-3, to=4-2]
	\arrow["2"{description}, from=6-3, to=5-4]
	\arrow["\begin{array}{c} \begin{smallmatrix}2\\1\end{smallmatrix} \end{array}"{description}, from=6-5, to=4-5]
	\arrow["3"{description}, from=6-5, to=5-4]
	\arrow["2"{description}, dashed, from=7-1, to=2-5]
	\arrow["3"{description}, from=7-1, to=5-1]
	\arrow["1"{description}, from=7-3, to=5-1]
	\arrow["\begin{array}{c} \begin{smallmatrix}3\\2\end{smallmatrix} \end{array}", from=7-3, to=6-3]
	\arrow["2"{description}, from=8-3, to=6-5]
	\arrow["1"{description}, from=8-3, to=7-1]
	\arrow["3", from=8-3, to=7-3]
\end{tikzcd}$}
Here every vertex is marked with the torsion pair $(\mathcal T\mid \mathcal F)$. The elements of the corresponding wide subcategories have been bolded and the summands of the support tilting modules have been underlined. The mutable interval corresponding to the first 2-cluster tilting object above is the simple interval at the vertex $(\begin{smallmatrix}3\\2\\1\end{smallmatrix},\begin{smallmatrix}2\end{smallmatrix},\begin{smallmatrix}3\\2\end{smallmatrix},\begin{smallmatrix}3\end{smallmatrix}\mid \begin{smallmatrix}1\end{smallmatrix},\begin{smallmatrix}2\\1\end{smallmatrix})$. 
Note that the torsion class and the torsion free class of this pair match the one we constructed from the 2-cluster tilting object. The mutable interval corresponding to the second 2-cluster tilting object is $[(\begin{smallmatrix}2\\1\end{smallmatrix},\begin{smallmatrix}2\end{smallmatrix}\mid \mathcal F)\leq(\mathcal T'\mid \begin{smallmatrix}1\end{smallmatrix})]$. The interval mutation corresponding to the almost 2-cluster tilting object above is given by $(B,I,A)$ where
\begin{itemize}
    \item $I = [(\mid \modulecat \Lambda)\leq(\mathcal T_I'\mid \begin{smallmatrix}1\end{smallmatrix})]$
    \item $B = [(\begin{smallmatrix}3\end{smallmatrix}\mid \mathcal F_B)\leq(\mathcal T_B'\mid \begin{smallmatrix}1\end{smallmatrix})]$
    \item $A = [(\mid \modulecat \Lambda)\leq(\mathcal T_A'\mid \begin{smallmatrix}1\end{smallmatrix},\begin{smallmatrix}3\end{smallmatrix})]$.
\end{itemize}
The corresponding augmented interval is $(I,\begin{smallmatrix}3\end{smallmatrix})$.

\section{Tamari lattices and geometric models}
    Tamari lattices are the lattices of torsion classes of linear type $A$ path algebras. They were proven to be fractionally Calabi-Yau by Rognerud in \cite{RognerudTamari}. Our proof is partially based on his approach.  We want to use this section to compare the two arguments and to relate the geometric model used by Rognerud to the geometric model Baur and Marsh introduced for the 2-cluster category in type $A$ in \cite{GeometricClusterCategory}.\\
    We start by giving an overview of Rognerud's proof. For more details see \cite{RognerudTamari} and \cite{RognerudNCP}. Rognerud introduced a family of intervals called \emph{exceptional intervals} and showed that the Serre functor permutes them up to shift. The exceptional intervals are exactly the mutable intervals in this case since both are given by intervals in the lattice of noncrossing partitions. There are two main differences between our proof and Rognerud's: the way the Serre functor is computed and the way the exceptional intervals are parametrised. The interval modules of all exceptional intervals have boolean antichains and this allowed Rognerud to compute the Serre functor applied to each of them using a version of \cref{serreOfBoolean}. Every Cambrian lattice that is not Tamari or type I admits a mutable interval with a non-boolean antichain and hence this part of Rognerud's argument can not be generalised directly. To parametrise the exceptional intervals Rognerud used noncrossing trees.
    \begin{Definition}
        A \emph{noncrossing tree} of type $A_n$ is a subset of the diagonals of an $n+2$ gon such that
        \begin{enumerate}
            \item no two diagonals cross,
            \item every vertex is incident to at least one diagonal and
            \item the diagonals do not include a cycle.
        \end{enumerate}
        Note that boundary edges are considered diagonals for this definition.
    \end{Definition}
    Noncrossing trees admit a canonical permutation called planar duality. Rognerud showed that it corresponds to the Serre functor on exceptional interval. Note that the planar dual of the planar dual of a noncrossing tree is not the original tree, but a rotation of it by $2\pi/(n+2)$
    \begin{Definition}
        The \emph{planar dual} of a noncrossing tree $T$ is constructed as follows: We add a new vertex on each boundary segment of the $n+2$-gon. Then $T$ cuts the $n+2$-gon into regions with each regions containing exactly one of the new vertices. Each edge of $T$ separates two regions. For each edge in $T$ we consider a new edge connecting the vertices in the regions it separates. The planar dual of $T$ consists of all of these new edges after rotation by $2\pi/2(n+2)$. The rotation moves the new vertices onto the old ones.
        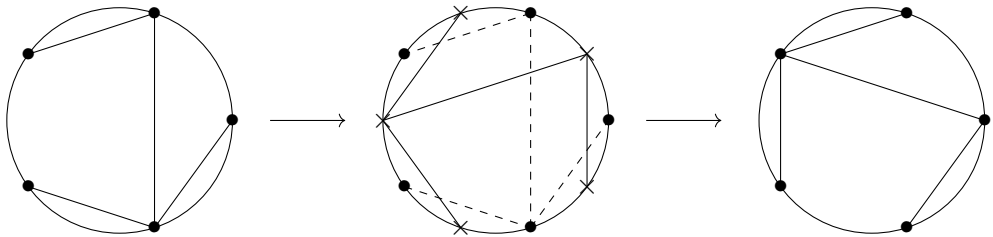
\begin{figure}[h]
        \centering
        \caption{Construction of the planar dual of a noncrossing tree.}
        \begin{tikzpicture}
			\draw (0,0) circle[radius=1.5] ;    
			\coordinate [label=center:$\bullet$] (A) at (72 : 1.5) ;
			\coordinate [label=center:$\bullet$] (B) at (144 : 1.5) ;
			\coordinate [label=center:$\bullet$] (C) at (216 : 1.5) ;
			\coordinate [label=center:$\bullet$] (D) at (288 : 1.5) ;
			\coordinate [label=center:$\bullet$] (E) at (0 : 1.5) ;
			\draw (A) -- (B) ;
			\draw (A) -- (D) ;
			\draw (C) -- (D) ;
			\draw (D) -- (E) ;

            \draw (5,0) circle[radius=1.5] ;    
			\coordinate [label=center:$\bullet$] (A2) at ($(72 : 1.5)+(5,0)$) ;
			\coordinate [label=center:$\bullet$] (B2) at ($(144 : 1.5)+(5,0)$) ;
			\coordinate [label=center:$\bullet$] (C2) at ($(216 : 1.5)+(5,0)$) ;
			\coordinate [label=center:$\bullet$] (D2) at ($(288 : 1.5)+(5,0)$) ;
			\coordinate [label=center:$\bullet$] (E2) at ($(0 : 1.5)+(5,0)$) ;
            \coordinate [label=center:$\times$] (F) at ($(108 : 1.5)+(5,0)$) ;
			\coordinate [label=center:$\times$] (G) at ($(180 : 1.5)+(5,0)$) ;
			\coordinate [label=center:$\times$] (H) at ($(252 : 1.5)+(5,0)$) ;
			\coordinate [label=center:$\times$] (I) at ($(324 : 1.5)+(5,0)$) ;
			\coordinate [label=center:$\times$] (J) at ($(36 : 1.5)+(5,0)$) ;
			\draw[dashed] (A2) -- (B2) ;
			\draw[dashed] (A2) -- (D2) ;
			\draw[dashed] (C2) -- (D2) ;
			\draw[dashed] (D2) -- (E2) ;
            \draw (J) -- (I) ;
			\draw (J) -- (G) ;
			\draw (F) -- (G) ;
			\draw (G) -- (H) ;

            \draw (10,0) circle[radius=1.5] ;    
			\coordinate [label=center:$\bullet$] (F2) at ($(72 : 1.5)+(10,0)$) ;
			\coordinate [label=center:$\bullet$] (G2) at ($(144 : 1.5)+(10,0)$) ;
			\coordinate [label=center:$\bullet$] (H2) at ($(216 : 1.5)+(10,0)$) ;
			\coordinate [label=center:$\bullet$] (I2) at ($(288 : 1.5)+(10,0)$) ;
			\coordinate [label=center:$\bullet$] (J2) at ($(0 : 1.5)+(10,0)$) ;
            \draw (J2) -- (I2) ;
			\draw (J2) -- (G2) ;
			\draw (F2) -- (G2) ;
			\draw (G2) -- (H2) ;

            \draw[->] (2,0) -- (3,0);
            \draw[->] (7,0) -- (8,0);
			\end{tikzpicture}
            \end{figure}
    \end{Definition}
    A geometric model for the 2-cluster category in type $A$ was given by Baur and Marsh in \cite{GeometricClusterCategory}. They realise the indecomposable objects in the 2-cluster category of type $A_n$ as non-boundary diagonals in a $2(n+2)$-gon that connect vertices of opposite parity. A sum of indecomposable objects is 2-rigid iff the corresponding diagonals do not intersect in their interior. In particular the 2-cluster tilting objects correspond to the quadrangulations of the $2(n+2)$-gon. The shift in the 2-cluster category corresponds to a rotation by $2\pi/2(n+2)$. One of the main results in this article identifies the shift in the 2-cluster category with the Serre permutation on augmented intervals and hence one would expect a bijection between quadrangulations and noncrossing trees that identifies this rotation with planar duality. Such a bijection is given in \cite{Stokes} as follows:
    \begin{Proposition}
        Let $Q$ be a quadrangulation of a $2(n+2)$-gon. We mark the vertices alternating with $\bullet$ and $\times$. Then every quadrangle contains a diagonal connecting two $\bullet$ vertices. The set of these diagonals forms a noncrossing tree in the $n+2$-gon given by the $\bullet$ vertices. This defines a bijection between quadrangulations and noncrossing trees which identifies rotation with planar duality.
    \end{Proposition}
    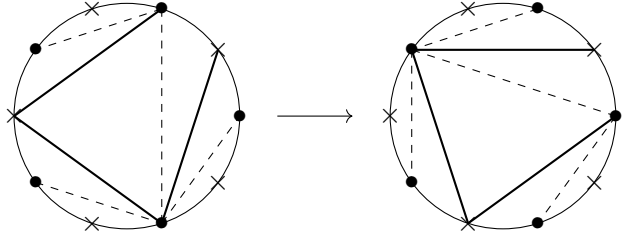
\begin{figure}[h]
        \centering
        \caption{A quadrangulation of a 10-gon and its rotation with the corresponding noncrossing trees drawn in dashed lines.}
    \begin{tikzpicture}
        \draw (0,0) circle[radius=1.5] ;    
			\coordinate [label=center:$\bullet$] (C) at (72 : 1.5) ;
			\coordinate [label=center:$\bullet$] (E) at (144 : 1.5) ;
			\coordinate [label=center:$\bullet$] (G) at (216 : 1.5) ;
			\coordinate [label=center:$\bullet$] (I) at (288 : 1.5) ;
			\coordinate [label=center:$\bullet$] (A) at (0 : 1.5) ;
            \coordinate [label=center:$\times$] (D) at (108 : 1.5) ;
			\coordinate [label=center:$\times$] (F) at (180 : 1.5) ;
			\coordinate [label=center:$\times$] (H) at (252 : 1.5) ;
			\coordinate [label=center:$\times$] (J) at (324 : 1.5) ;
			\coordinate [label=center:$\times$] (B) at (36 : 1.5) ;
			\draw[thick] (B) -- (I) ;
			\draw[thick] (C) -- (F) ;
			\draw[thick] (F) -- (I) ;
            \draw[dashed](I) -- (A);
            \draw[dashed](C) -- (I);
            \draw[dashed](C) -- (E);
            \draw[dashed](G) -- (I);

            \draw[->] (2,0) -- (3,0);

            \draw (5,0) circle[radius=1.5] ;    
			\coordinate [label=center:$\bullet$] (D2) at ($(72 : 1.5)+(5,0)$) ;
			\coordinate [label=center:$\bullet$] (F2) at ($(144 : 1.5)+(5,0)$) ;
			\coordinate [label=center:$\bullet$] (H2) at ($(216 : 1.5)+(5,0)$) ;
			\coordinate [label=center:$\bullet$] (J2) at ($(288 : 1.5)+(5,0)$) ;
			\coordinate [label=center:$\bullet$] (B2) at ($(0 : 1.5)+(5,0)$) ;
            \coordinate [label=center:$\times$] (E2) at ($(108 : 1.5)+(5,0)$) ;
			\coordinate [label=center:$\times$] (G2) at ($(180 : 1.5)+(5,0)$) ;
			\coordinate [label=center:$\times$] (I2) at ($(252 : 1.5)+(5,0)$) ;
			\coordinate [label=center:$\times$] (A2) at ($(324 : 1.5)+(5,0)$) ;
			\coordinate [label=center:$\times$] (C2) at ($(36 : 1.5)+(5,0)$) ;
			\draw[thick] (B2) -- (I2) ;
			\draw[thick] (C2) -- (F2) ;
			\draw[thick] (F2) -- (I2) ;
            \draw[dashed](B2) -- (F2);
            \draw[dashed](D2) -- (F2);
            \draw[dashed](F2) -- (H2);
            \draw[dashed](B2) -- (J2);
    \end{tikzpicture}
    \end{figure}

\section{Appendix by Rene Marczinzik: Periodic Serre formal algebras for incidence algebras of lattices}
In this appendix we want to motivate the classification of periodic Serre formal incidence algebras of lattices, which was one of the motivating factors for this article. It seems the classification of periodic Serre formal incidence algebras is a much more tractable problem compared to the classification of fractionally Calabi-Yau incidence algebras of lattices as computer experiments suggest.
We will introduce a combinatorial version of periodic Serre formal algebras and pose the conjecture that for incidence algebras of lattices those two notions coincide.
Let $A=KQ/I$ be a quiver algebra of finite global dimension with $n$ simple modules. Let $\omega$ be the Cartan matrix of $A$ given as the $n \times n$-matrix with entries $\omega_{i,j}=\dim e_j A e_i$.
The Coxeter matrix $C$ of $A$ is given by $-\omega^{T} \omega^{-1}$ and we have that $C$ sends the dimension vector of indecomposable projective modules $P_i$ to minus the dimension vector of the indecomposable injective module $I_i$. We refer for example to \cite[Chapter III.3]{ASS} for more information and properties of the Coxeter matrix.
We denote by $[M]$ elements in the Grothendieck group $K_0(A)$ given by a module $M$.  
Recall that an $A$-module $M$ is called \emph{perfect} if it has finite projective dimension $n$ and $\Ext_A^i(M,A) \neq 0$ iff $i=n$.
\begin{Proposition} \label{Coxeterpropo}
Let $A$ be an algebra of finite global dimension and $M$ a perfect $A$-module of projective dimension $n$. Then $C ([M])= (-1)^{n+1} [\tau_n(M)]$, where we set here $\tau_0:=\nu$ the Nakayama functor.
\end{Proposition}
\begin{proof}
For $n=0$, by definition of the Coxeter matrix we have $C[P]=-[\nu(P)]$ for a projective module $P$ and the Nakayama functor $\nu$. Thus assume $n>0$ and let 
\begin{align} \label{projresM}
0 \rightarrow P_n \rightarrow P_{n-1} \rightarrow \cdots \rightarrow P_0 \rightarrow M \rightarrow 0
\end{align}
be a minimal projective resolution of $M$.
Applying first the functor $\Hom_A(-,A)$ and then the duality $D$ and using that $M$ is perfect gives the exact sequence:
\begin{align}\label{secondexactseq}
0 \rightarrow D\Ext_A^n(M,A) \rightarrow D \Hom_A(P_n,A) \rightarrow \cdots \rightarrow \Hom_A(P_0,A) \rightarrow 0.
\end{align}
Since this also exact sequence coincides with the exact sequence coming from the definition of $\tau_n(M)=\tau \Omega^{n-1}(M)$, we have $D \Ext_A^n(M,A) \cong \tau_n(M)$.
From the exact sequence \ref{projresM} we obtain the following equality in the Grothendieck group:
$$[M]=\sum\limits_{i=0}^{n}{(-1)^i [P_i]}$$
and from the exact sequence \ref{secondexactseq} we obtain the equality:
$$[D\Ext_A^n(M,A)]= \sum\limits_{i=0}^{n}{(-1)^i [D \Hom_A(P_i,A)]}.$$
Now using the definition of the Coxeter matrix and combining those identities gives: $$C([M])=C(\sum\limits_{i=0}^{n}{(-1)^i [P_i]})=-\sum\limits_{i=0}^{n}{(-1)^i [D\Hom_A(P_i,A)]}$$
$$= -(-1)^n \sum\limits_{r=0}^{n}{(-1)^r [D \Hom_A(P_{n-r},A)]}=(-1)^{n+1} [D \Ext_A^n(M,A)].$$ \end{proof}
Motivated by the previous proposition, we give the following definition:
\begin{Definition}
Let $A$ be an algebra of finite global dimension with Coxeter matrix $C$ and $n$ points labeled by $1,...,n$.
We call an indecomposable injective $A$-module $I_i$ \emph{combinatorially Serre formal} if 
there is a minimal $n_i \geq 0$ such that $C^{n_i}([I_i])=[P_{\pi(i)}]$, that is the $n_i$-th power of $C$ applied to $[I_i]$ gives the dimension vector of an indecomposable projective module $P_{\pi(i)},$ and furthermore $C^k [I_i]$ is always a weakly positive dimension vector or a weakly negative dimension vector for all $0 \leq k \leq n_i$. We call $A$ \emph{combinatorially Serre formal} if every indecomposable injective $A$-modules are combinatorially Serre formal and the map $\pi: \{1,..,n\} \rightarrow \{1,...,n\}$ is a bijection, which we all the \emph{Serre permutation} in that case.
\end{Definition}

The next corollary is a direct consequence of the definition of periodic Serre formal algebras, Proposition \ref{Coxeterpropo} and the fact that the Coxeter transformation is up to a sign the linear algebraic version of the derived Nakayama functor:
\begin{Corollary}
A periodic Serre formal algebra of finite global dimension is combinatorially Serre formal. 
\end{Corollary}
We pose the following conjecture:
\begin{Conjecture}
Let $A=kL$ be the incidence algebra of a lattice $L$.
Then $A$ is combinatorially Serre formal if and only if $A$ is periodic Serre formal.
\end{Conjecture}
We verified this conjecture for all lattices with at most 12 elements. 
The main result in this article showed that Cambrian lattices are periodic Serre formal. We pose it as a problem to give a general classification:
\begin{Problem}
Classify the incidence algebras $KL$ with $L$ a lattice such that $KL$ is periodic Serre formal.
\end{Problem}
The next result gives the classification in the distributive case using the main results of \cite{GKKM}:
\begin{Proposition} \label{divisorlatticecase}
Let $KL$ be the incidence algebra of a distributive lattice $L$. Then $KL$ is periodic Serre formal if and only if $L$ is a divisor lattice.
\end{Proposition}
\begin{proof}
A divisor lattice is periodic Serre formal as periodic Serre formal is invariant under taking tensor products and divisor lattices are by definition tensor products of hereditary linear Nakayama algebras.
Conversely, note that in a periodic Serre formal algebra $A$ the indecomposable injective (left and right) modules $I$ are all perfect, that is $Ext_A^i(I,A)$ iff $i=\pdim I$. Now by \cite[Theorem 5.4]{GKKM} a distributive lattice with all indecomposable injective (left and right) modules being perfect is a divisor lattice. 
\end{proof}

While Cambrian lattices are always semidistributive in the sense of \cite{semidistRST}, it is in general not true that periodic Serre formal lattices are semidistributive as the next example shows:
\begin{Example}
In this example we will use that the derived Nakayama functor applied iteratively to an indecomposable injective $A$-module in a periodic Serre formal algebra $A$ can be calculated by computing the higher Auslander-Reiten translate $\tau_n$ where $n$ denotes the projective dimension of the relevant module.
Let $L$ be the lattice with Hasse diagram 
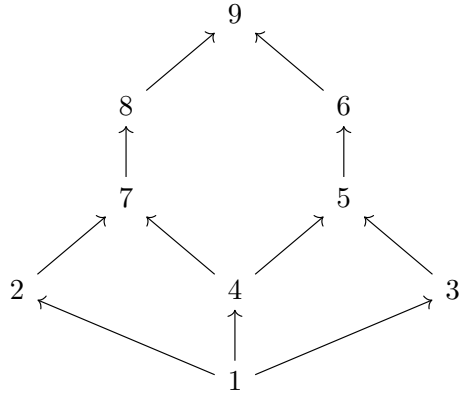
\[\begin{tikzcd}
	&& 9 \\
	& 8 && 6 \\
	& 7 && 5 \\
	2 && 4 && 3 \\
	&& 1
	\arrow[from=2-2, to=1-3]
	\arrow[from=2-4, to=1-3]
	\arrow[from=3-2, to=2-2]
	\arrow[from=3-4, to=2-4]
	\arrow[from=4-1, to=3-2]
	\arrow[from=4-3, to=3-2]
	\arrow[from=4-3, to=3-4]
	\arrow[from=4-5, to=3-4]
	\arrow[from=5-3, to=4-1]
	\arrow[from=5-3, to=4-3]
	\arrow[from=5-3, to=4-5]
\end{tikzcd}\]
Let $A$ denote the incidence algebra of $L$.
We give the applications of the higher Auslander-Reiten on the indecomposable injective $A$-modules and leave the verifications to the reader such as that all involved modules are perfect.
By symmetry, it is enough to consider the indecomposable injective modules 
$I(1), I(3), I(4),$ 
$I(5), I(6)$ and $I(9)$. Let $N$ denote the unique indecomposable $A$-module with dimension vector $[ 0, 0, 0, 1, 1, 0, 1, 0, 0 ]$.
\begin{enumerate}
    \item $I(1) \xrightarrow{\tau_2} N \xrightarrow{\tau_2} P(9)$.
    \item $I(3) \xrightarrow{\tau_2} S(7) \xrightarrow{\tau} S(8) \xrightarrow{\tau} P(3)$.
    \item $I(4) \xrightarrow{\tau_2} P(4)$.
    \item $I(5) \xrightarrow{\tau_2} P(7)$.
    \item $I(6) \xrightarrow{\tau} S(2) \xrightarrow{\tau} M_{[4,7]} \xrightarrow{\tau_2} P(6)$.
    \item $I(9) \cong P(1)$.
\end{enumerate}
We see that the Serre permutation on $L$ is given as the permutation:
\[
\left(
\begin{array}{ccccccccc}
1&2&3&4&5&6&7&8&9\\
9&2&3&4&7&6&5&8&1
\end{array}
\right)
\]
We remark that the module $N$ is not an interval module, so this example also shows that for periodic Serre formal lattices the higher Auslander-Reiten translates of indecomposable injective modules do not have to be interval modules.
\end{Example}
It might be interesting to compare the Serre permutation to the classical rowmotion map for semidistributive lattices (see \cite{B}) or the more general echelonmotion bijection for more general classes of lattices (see for example \cite{DJMSSTW}). Are there interactions between those two permutations?

We give the final remark that in forthcoming work we will develop fast algorithms to decide whether a given algebra $A$ is twisted fractionally Calabi-Yau using replicated algebras and trivial extension algebras, see \cite{CDIM,CIM2}, and also to test whether $A$ is periodic Serre formal using Proposition \ref{Coxeterpropo} as a first quick test.
In our experience it seems the behaviour of the fractional Calabi-Yau property for incidence algebras of lattices and more general algebras is very chaotic compared to the  periodic Serre formal property. For example already for Nakayama algebras, the classification of twisted fractional Calabi-Yau algebras is wide open with not even an existing conjecture, while the classification of periodic Serre formal linear Nakayama algebras is easy, see \cite{ProductsofSerreFormal}.

\bibliographystyle{alpha}
\bibliography{references}

\end{document}